\documentclass{amsart}

\usepackage{amsmath}
\usepackage{amsthm}
\usepackage{amssymb}
\usepackage{graphicx}
\usepackage{bm}
\usepackage{cleveref}
\usepackage{enumerate}
\usepackage{multirow}
\usepackage{makecell}
\usepackage{mathtools}
\newcommand{\Mb}[1]{\left[{#1}\right]}
\newcommand{\Ma}[1]{\left\langle{#1}\right\rangle}

\newcommand{\vF}{\bm{F}}

\newcommand{\vx}{\bm{x}}

\newcommand{\vk}{\bm{k}}

\newcommand{\vv}{\bm{v}}
\newcommand{\vz}{\bm{z}}
\newcommand{\vu}{\bm{u}}

\newcommand{\va}{\bm{a}}
\newcommand{\vf}{\bm{f}}
\newcommand{\vg}{\bm{g}}
\newcommand{\vs}{\bm{s}}
\newcommand{\vn}{\bm{n}}

\newcommand{\vr}{\bm{r}}
\newcommand{\vzero}{\bm{0}}
\newcommand{\vone}{\bm{1}}
\newcommand{\veps}{\bm{\varepsilon}}
\newcommand{\valpha}{\bm{\alpha}}
\newcommand{\vbeta}{\bm{\beta}}

\newcommand{\mA}{\mathsf{A}}
\newcommand{\mM}{\mathsf{M}}

\newcommand{\mQ}{\mathsf{Q}}
\newcommand{\mK}{\mathsf{K}}
\newcommand{\mI}{\mathsf{I}}
\newcommand{\mP}{\mathsf{P}}

\newcommand{\mU}{\mathsf{U}}
\newcommand{\mV}{\mathsf{V}}

\newcommand{\mLambda}{\mathsf{\Lambda}}

\newcommand{\real}{\mathbb{R}}
\newcommand{\complex}{\mathbb{C}}

\newcommand{\znat}{\mathbb{Z}}

\newtheorem{theorem}{Theorem}[section]
\newtheorem{lemma}{Lemma}[section]

\theoremstyle{definition}
\newtheorem{remark}{Remark}[section]
\newtheorem{example}{Example}[section]
\newtheorem{definition}{Definition}[section]

\DeclareMathOperator{\linspan}{span}
\DeclareMathOperator{\diag}{diag}
\DeclareMathOperator{\range}{range}

\begin{document}

\title{Periodic particle arrangements using
standing acoustic waves}

\author{Fernando Guevara Vasquez$^{1}$ and China~Mauck$^{1}$}

\address{$^{1}$Mathematics Department, University of Utah, Salt Lake City UT
84112, USA}

\thanks{FGV and CM contributed to both the theoretical results and
the numerical experiments. The authors acknowledge support from Army Research Office
Contract No. W911NF-16-1-0457.}



%\corres{Fernando Guevara Vasquez\\
%\email{fguevara@math.utah.edu}}

%\maketitle
\begin{abstract}
We determine crystal-like materials that can be fabricated by
using a standing acoustic wave to arrange small particles in a non-viscous
liquid resin, which is cured afterwards to keep the particles in the desired
locations. For identical spherical particles with the same physical
properties and small compared to the wavelength, the locations  where
the particles are trapped correspond to the minima of an acoustic
radiation potential which describes the net forces that a particle is
subject to. We show that the global minima of spatially periodic
acoustic radiation potentials can be predicted by the eigenspace of a
small real symmetric matrix corresponding to its smallest eigenvalue. We
relate symmetries of this eigenspace to particle arrangements
composed of points, lines or planes. Since waves are used to generate
the particle arrangements, the arrangement's periodicity is limited to certain
Bravais lattice classes that we enumerate in two and three dimensions. % 148 words
\end{abstract}

%\tableofcontents % to be removed in final version
\maketitle

%%%%%%%%%%%%%%%%%%%%%%%%%%%%%%%%%%%%%%%%%%%%%%%%%%%%%%%%%%%%%%%%%%%%%%%%
%\begin{fmtext}
\section{Introduction}
\label{sec:intro}

We are interested in characterizing the possible periodic or
crystal-like materials that can be fabricated by using ultrasound
directed self-assembly \cite{Greenhall:2015:UDS,Prisbrey:2017:UDS}. In
this fabrication method, a liquid resin containing small particles is
placed in a reservoir that is lined by ultrasound transducers. By
operating the transducers at a fixed frequency, a standing acoustic wave
is generated in the liquid and drives the particles to certain
locations.  For example, when the particles are neutrally buoyant and
less compressible than the surrounding fluid, the particles tend to go
to the wave nodes (zero amplitude locations), as we explain later.  Once
the particles are in the desired positions, the resin is cured (i.e.
hardened with light, a curing agent, etc.) and we obtain a material with
inclusions placed in a periodic fashion.  
%\end{fmtext}\maketitle

To give a rough idea, one could use this technique to fabricate a
crystal-like material (essentially a grating) that selectively reflects
millimetre waves, i.e. electromagnetic waves with wavelengths of the
order of $1$mm to $1$cm. Indeed, one could achieve this by placing
sub-wavelength metallic particles periodically inside a dielectric resin
matrix. If we assume the speed of sound in the resin matrix is
$1500$ms$^{-1}$, we can expect the particles to be separated by between
$5\mu$m and $50\mu$m provided the transducer operating frequency is
between $150$kHz and $1.5$MHz. This rough estimate is based on a spacing
of half an ultrasound wavelength that we rigorously justify later. A
simple motivation, for now, is to observe that in one dimension there
are actually two nodes of the wave per wavelength. We emphasize though
that our analysis is valid in the case where the particles do not
cluster about the wave nodes and holds in two or three dimensions. We
note that acoustically configurable crystals were already considered in
\cite{Caleap:2014:ATC,Silva:2019:PPU}. Our work shows what are the
theoretical limits of such approaches and gives explicit manipulation
strategies.  We also mention that acoustic manipulation of multiple particles has
other applications to tissue engineering \cite{Armstrong:2018:EAM},
fabrication of laminates \cite{Greenhall:2017:PME} and as acoustic
tweezers \cite{Marzo:2019:HAT}.

For our study we assume that acoustic waves at a fixed frequency $f$
propagate in a fluid. The pressure and fluid velocity
fields have the form $\widetilde{p}(\vx,t) = \Re(\exp[-i\omega t]
p(\vx))$ and $\widetilde{\vv}(\vx,t) = \Re(\exp[-i\omega t] \vv(\vx))$
at position $\vx \in \real^d$ ($d=2$ or $3$) and time $t$. Here $\omega = 2\pi f$ is the angular frequency and $\Re$
denotes the real part of a complex quantity. The time domain pressure
and fluid velocity are related by 
\begin{equation} 
\begin{aligned}
 \widetilde{p}_t + \kappa_0 \nabla \cdot \widetilde{\vv} &= 0,~
 \text{and}\\
 \rho_0 \widetilde{\vv}_t + \nabla \widetilde{p} &=0,
\end{aligned}
\label{eq:weq1}
\end{equation}
where $\kappa_0$ is the compressibility of the fluid and $\rho_0$ its
mass density. Clearly the frequency domain pressure $p$ is a complex
valued solution to the Helmholtz equation $\Delta p + k^2 p = 0$, where
the wavenumber is $k = \omega / c$, and $c = \sqrt{\kappa_0/c_0}$ is the
velocity of propagation of acoustic waves in the fluid (see e.g.
\cite{Colton:1998:IAE}). Notice that the pressure and velocities in the
frequency domain are related by $\vv = (i\omega\rho_0)^{-1} \nabla p$.

%%%%%%%%%%%%%%%%%%%%%%%%%%%%%%%%%%%%%%%%%%%%%%%%%%%%%%%%%%%%%%%%%%%%%%%%
\subsection{The acoustic radiation potential}
\label{sec:arp:def}
Small (compared to the wavelength) particles in a fluid are subject to
an acoustic radiation force
\cite{Gorkov:1962:FAS,King:1934:ARP,Settnes:2012:FAS}. It is
convenient to study the average of this force over a time period $T =
1/f$. To be more precise, the 
time average over a period of $T-$periodic function $g$ is
\[
 \langle g \rangle = \frac{1}{T}\int_0^T g(t) dt.
\]
The net acoustic radiation force experienced by a small particle and averaged
over a time period is proportional to $\vF = - \nabla \psi$, where
$\psi(\vx)$ is the so-called acoustic radiation potential. For small spherical
particles of fixed size, the acoustic radiation potential is given
by\footnote{We denote by $|\vv|$ the Euclidean norm of a vector $\vv$.}
\begin{equation}
 \label{eq:arp1}
  \psi = \mathfrak{f}_1 \frac{\kappa_0}{2} \Ma{|\widetilde{p}|^2} -
  \mathfrak{f}_2 \frac{3 \rho_0}{4}
 \Ma{|\widetilde{\vv}|^2},
\end{equation}
where $\mathfrak{f}_1 = 1 - (\kappa_p/\kappa_0)$ and $\mathfrak{f}_2 =
2(\rho_p - \rho_0)/(2\rho_p + \rho_0)$ are non-dimensional constants
depending on the compressibilities  of the particle and the fluid
($\kappa_p$ and $\kappa_0$, respectively) and on their mass densities
($\rho_p$ and $\rho_0$, respectively). 

Clearly, particles tend to {\em cluster at the minima of the potential
$\psi$}.  The acoustic radiation potential can be rewritten in terms of
the frequency domain pressure as follows\footnote{Recall that if
$\widetilde{\vf}(t) = \Re (\vf \exp[-i\omega t])$, then its time average
is $\langle |\widetilde{\vf}(t)|^2 \rangle = |\vf|^2/2$.}
\begin{equation}
  \label{eq:arp}
  \psi = \mathfrak{a} |p|^2 - \mathfrak{b} |\nabla p|^2 =
  \begin{bmatrix} p \\ \nabla p \end{bmatrix}^*
  \begin{bmatrix}
   \mathfrak{a} & \\ & -\mathfrak{b} \mI_d
  \end{bmatrix}
  \begin{bmatrix} p \\ \nabla p \end{bmatrix},
\end{equation}
where $\mathfrak{a} = \mathfrak{f}_1 \kappa_0/4$,  $\mathfrak{b}=\mathfrak{f}_2 3 / (8\rho_0
\omega^2)$, $\mI_d$ is the $d\times d$ identity matrix and $*$ denotes the
conjugate transpose. The key observation here is that
for fixed $\vx$, we can think of the acoustic radiation potential as a
quadratic form in $p$ and $\nabla p$. We make no
particular assumption on the signs of $\mathfrak{a}$ and $\mathfrak{b}$, as they depend on
the physical properties of the particles and the fluid.

\begin{remark}
In general the acoustic radiation potential depends also on the size and
shape of the particles, thus by using \eqref{eq:arp1} to predict where
the particles cluster, we are neglecting the effect of the particle's size
and shape.
\end{remark}

%%%%%%%%%%%%%%%%%%%%%%%%%%%%%%%%%%%%%%%%%%%%%%%%%%%%%%%%%%%%%%%%%%%%%%%%
\subsection{Spatially periodic acoustic waves}
\label{sec:paw}
To get a periodic arrangement of particles we take the pressure field
$p$ to be periodic. This can be achieved by taking a superposition of
plane waves with wavevectors $\vk_1,\ldots,\vk_d$ in dimension $d$,
that is
\begin{equation}
 p(\vx;\vu) = \sum\limits_{j=1}^d \alpha_j\exp[i\vk_j\cdot
 \vx] + \beta_j\exp[-i\vk_j\cdot \vx],
 \label{eq:amplitudes}
\end{equation}
where $\vu = [\alpha_1,\ldots,\alpha_d,\beta_1,\ldots,\beta_d]^T$ are
complex amplitudes, the wavenumber is $k = |\vk_j| = 2\pi/\ell$ and the
wavelength is $\ell = c/f$. If the wavevectors $\vk_j$ form a basis of
$\real^d$, which is assumed hereinafter, we can define the {\em lattice}
\cite{Kittel:2005:ISP} vectors $\va_1,\cdots,\va_d$
by introducing the $d\times d$ real matrices $\mA \equiv
[\va_1,\cdots,\va_d]$ and $\mK \equiv [\vk_1, \cdots , \vk_d]$ which are
related by
\begin{equation}
\label{eq:dual}
\mA = 2\pi \mK^{-T}.
\end{equation}
The field $p$ is $\mA-$periodic since 
\begin{equation}
 p( \vx + \mA \vn; \vu) = p(\vx; \vu)~\text{for any}~\vx \in
 \real^d,~\vn \in \znat^d,~\text{and}~\vu \in \complex^{2d}.
\end{equation}
Experimentally we expect that an acoustic pressure field similar to the
real part of \eqref{eq:amplitudes} can be obtained far away from planar
ultrasound transducers with normal orientations $\vk_j$ by using the 
complex amplitudes $\vu$ to determine the amplitudes and phases of the
voltage driving the transducers (see \cref{fig:setup} for an
illustration). Because of this we call $\vu$ {\em transducer
parameters}. We do not include the relation between $\vu$ and an actual
transducer's operating voltages as it is out of the scope of the present
study.

\begin{figure}
 \begin{center}
  \includegraphics[width=0.3\textwidth]{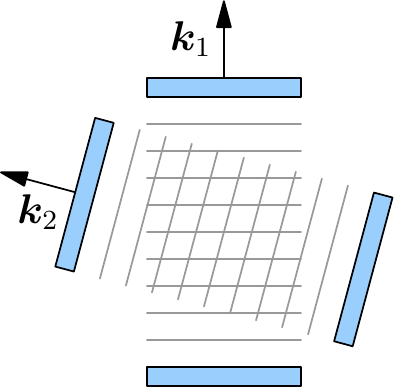}
 \end{center}
 \caption{A possible arrangement of ultrasound transducers (in blue) to
 generate fields close to \eqref{eq:amplitudes} in two dimensions.}
 \label{fig:setup}
\end{figure}

Since $p$ is $\mA-$periodic, its associated acoustic
radiation potential $\psi$ must also be $\mA-$periodic. In \cref{sec:arp}
we show that the extrema of the acoustic radiation potential can be
predicted by the maximum and minimum eigenvalues of a $2d \times 2d$
real symmetric matrix. We show that the level-sets of the acoustic
radiation potential at values equal to the eigenvalues of this matrix
are determined by the associated eigenspace.  Note that the relation
between the acoustic radiation potential and an eigendecomposition was
already exploited in \cite{Greenhall:2015:UDS, Prisbrey:2017:UDS} to
minimize the acoustic radiation potential at a set of points. Here we
are in a especial case where we
can find the minimizers {\em explicitly} and we give sufficient
conditions for these to be isolated points, lines or planes arranged
$\mA-$periodically. The lattice vectors of the possible particle
patterns must have reciprocal vectors with the same length. This limits the
possible classes of crystallographic symmetries (Bravais lattices) that can be
achieved with this method. We explore this limitation in two and three
dimensions in \cref{sec:lims}. We summarize our findings and questions
that were left open in \cref{sec:summary}.

%%%%%%%%%%%%%%%%%%%%%%%%%%%%%%%%%%%%%%%%%%%%%%%%%%%%%%%%%%%%%%%%%%%%%%%%
\section{Study of the acoustic radiation potential}
\label{sec:arp}

The key to our study is to write the acoustic radiation potential as a
quadratic form of the amplitudes $\vu$ driving the plane waves
(\cref{sec:quad}). We observe that spatial shifts are equivalent to a
similarity transformation of the matrix defining the quadratic form.
Therefore it suffices to study the acoustic radiation potential at the
origin (\cref{sec:eigen}), where we can write the eigendecomposition of the
associated matrix explicitly. Its eigenvalues give simple bounds on the
acoustic radiation potential at constant power. In particular the global
minimum values must correspond to the smallest eigenvalue of the
associated matrix (\cref{sec:bounds}). Then in \cref{sec:values} we
study the possible level-sets of the acoustic radiation potential (at
constant power) that are equal to one of the eigenvalues of the
associated matrix. These particular level-sets may be composed of lattices (with
anywhere between $2$ and $2^d$ points per primitive cell), lines or
planes, as is summarized in \cref{thm:summary}.

%%%%%%%%%%%%%%%%%%%%%%%%%%%%%%%%%%%%%%%%%%%%%%%%%%%%%%%%%%%%%%%%%%%%%%%%
\subsection{The acoustic radiation potential as a quadratic form}
\label{sec:quad}
For a superposition of plane waves of the form \eqref{eq:amplitudes}, the acoustic radiation
potential at a point $\vx$ can be written as
\begin{equation}
  \label{eq:arpquad}
  \psi(\vx;\vu) = \vu^* \mQ(\vx) \vu,
\end{equation}
where the $2d\times 2d$ Hermitian matrix $\mQ(\vx)$ is defined by
\begin{equation}
  \label{eq:Q}
  \mQ(\vx) = \mM(\vx)^* \begin{bmatrix} \mathfrak{a} &\\ &
  -\mathfrak{b}\mI_d \end{bmatrix} \mM(\vx),
\end{equation}
and we have used the $(d+1) \times 2d$ complex matrix $\mM(\vx)$ that is
given by
\begin{equation}
  \label{eq:M}
  \mM(\vx) = [\mM_+(\vx)\, \mM_-(\vx) ]
  ~\text{and}~
  \mM_{\pm}(\vx) =
  \begin{bmatrix}
   \exp[\pm i \vx^T \mK]\\ \mK \diag(\pm i\exp[\pm i \vx^T \mK])
   \end{bmatrix},
\end{equation}
where the exponential of a vector is understood componentwise and
$\diag([a_1,\ldots,a_n])$ is the matrix with diagonal elements
$a_1,\ldots,a_n$.

We remark that a translation in $\vx$ is equivalent to a {\em unitary}
similarity transformation of $\mQ(\vx)$. To see this consider a
point $\vx \in \real^d$, and write $\vx = \vx_0 + \veps$. Notice
that
\begin{equation}
  \label{eq:space-phase}
  \psi(\vx_0+\veps;\vu) =
  \psi(\vx_0; \exp[ i [\mK, -\mK]^T \veps ] \odot \vu),
\end{equation}
where $\odot$ is the componentwise or Hadamard product of two vectors.
Indeed, write\footnote{We use semicolons to stack vectors
i.e. $[\valpha;\vbeta] \equiv [\valpha^T,\vbeta^T]^T$, for vectors
$\valpha,\vbeta$.} $\vu =
[\valpha;\vbeta]$, $\valpha,\vbeta \in \complex^d$, and observe that a
spatial shift $\veps$ is equivalent to changing the phase of $\alpha_j$
by $\vk_j\cdot\veps$ and of $\beta_j$ by $-\vk_j\cdot\veps$. Therefore
$\mQ(\vx_0+\veps)$ and $\mQ(\vx_0)$ are related by the similarity
transformation
\begin{equation}
  \label{eq:similarity}
  \mQ(\vx_0+\veps) =
   \diag(\exp[ i [\mK, -\mK]^T \veps ])^*
   \mQ(\vx_0)
   \diag(\exp[ i [\mK, -\mK]^T \veps ]).
\end{equation}
A practical consequence of \eqref{eq:similarity} is that we can study
the acoustic radiation potential at a particular point $\vx_0$ and use
\eqref{eq:space-phase} or \eqref{eq:similarity} to deduce its behaviour
everywhere else. For simplicity we take $\vx_0 = \vzero$.

%%%%%%%%%%%%%%%%%%%%%%%%%%%%%%%%%%%%%%%%%%%%%%%%%%%%%%%%%%%%%%%%%%%%%%%%
\subsection{The acoustic radiation potential at the origin}
\label{sec:eigen}
After a straightforward calculation we get that
\begin{equation}
\label{eq:Q0}
 \mQ(\vzero) = \mathfrak{a} \vone\vone^T - \mathfrak{b} \begin{bmatrix}\mK\\-\mK\end{bmatrix}
 \begin{bmatrix}\mK\\-\mK\end{bmatrix}^T,
\end{equation}
where $\vone$ is a vector of all ones of appropriate dimension. The
following \cref{lem:Q0} gives the eigendecomposition of $\mQ(\vzero)$
explicitly in terms of  that of $\mK\mK^T$. 
\begin{lemma}
 \label{lem:Q0}
 Let $\sigma_1 \geq \sigma_2 \geq \ldots \geq \sigma_d > 0$ be the
 eigenvalues of $\mK\mK^T$ sorted in decreasing order and including
 multiplicity, and let
 $\{\vu_1,\ldots,\vu_d\}$ be a corresponding real orthonormal
 basis of eigenvectors\footnote{Alternatively, the $\sigma_j$ are the squares of the singular values of
 $\mK$ and the $\vu_j$ are its right singular vectors.}. Let
 $\{\vz_1,\ldots,\vz_{d-1}\}$ be a real orthonormal
 basis for $\vone^\perp$. Then 
 $\mQ(\vzero)$ admits the eigendecomposition $\mQ(\vzero) = \mU^T
 \mLambda \mU$, where $\mLambda$ is the $2d\times 2d$ diagonal matrix
 \begin{equation}
  \label{eq:Lambda}
  \mLambda = \diag([ 2\mathfrak{a}d, 0,\ldots,0,
  -2\mathfrak{b}\sigma_1,\ldots,-2\mathfrak{b}\sigma_d]),
 \end{equation}
 and $\mU$ is the real $2d \times 2d$ orthonormal matrix
 \begin{equation}
  \label{eq:U}
  \mU = \frac{1}{\sqrt{2}}
  \begin{bmatrix}
   \vone/\sqrt{d} & \vz_1 & \ldots & \vz_{d-1} & \vu_1 & \ldots & \vu_d\\
   \vone/\sqrt{d} & \vz_1 & \ldots & \vz_{d-1} & -\vu_1 & \ldots & -\vu_d
  \end{bmatrix}.
 \end{equation}
\end{lemma}

Before proceeding to the proof of \cref{lem:Q0}, it is useful to
introduce the decomposition $\complex^{2d} = H_+ \oplus H_-$ where
$H_\pm = \{ [\valpha;\pm\valpha] ~|~ \valpha \in \complex^d \}$. The
unitary matrices
\begin{equation}
 \mV_\pm = \frac{1}{\sqrt{2}}\begin{bmatrix} \mI_d \\ \pm \mI_d \end{bmatrix}
\end{equation}
are such that $H_\pm = \range(\mV_\pm)$. Orthogonal 
projectors onto $H_\pm$ are given by $\mP_\pm = \mV_\pm \mV_\pm^T$.

\begin{proof}
From our assumption that the wavevectors $\vk_1,\ldots,\vk_d$ form a
basis, it is clear that the smallest eigenvalue of $\mK \mK^T$ should be
positive. Now notice that we have
\begin{equation}
\begin{bmatrix}
 \mV_+ \mV_-
\end{bmatrix}^T
 \mQ(\vzero)
\begin{bmatrix}
 \mV_+ \mV_-
\end{bmatrix} 
=
\begin{bmatrix}
2\mathfrak{a}\vone \vone^T &\\
&-2\mathfrak{b}\mK\mK^T
\end{bmatrix}.
\end{equation}
Using the eigendecompositions of $\vone\vone^T$ and $\mK\mK^T$ readily
gives that of $\mQ(\vzero)$.
\end{proof}

%%%%%%%%%%%%%%%%%%%%%%%%%%%%%%%%%%%%%%%%%%%%%%%%%%%%%%%%%%%%%%%%%%%%%%%%
\subsection{Bounds on the acoustic radiation potential}
\label{sec:bounds}
A simple consequence of the acoustic radiation potential being
a quadratic form \eqref{eq:arpquad} is that for any transducer
parameters $\vu \in \complex^{2d}$ and positions $\vx \in \real^d$, we
have the bounds
\begin{equation}
 \lambda_{\min}(\vx) |\vu|^2 \leq \psi(\vx;\vu) \leq
 \lambda_{\max}(\vx) |\vu|^2,
\end{equation}
where $\lambda_{\max,\min}(\vx) = \lambda_{\max,\min}(\mQ(\vx))$ are the
minimum and maximum eigenvalues of $\mQ(\vx)$.
However we observed in \eqref{eq:similarity} that for arbitrary $\vx$,
$\mQ(\vx)$ is unitarily similar to $\mQ(\vzero)$, thus the eigenvalues
of $\mQ(\vx)$ do not depend on $\vx$. This immediately gives
the bound
\begin{equation}
 \lambda_{\min}(\vzero) |\vu|^2 \leq \psi(\vx;\vu) \leq
 \lambda_{\max}(\vzero) |\vu|^2,
 \label{eq:bound}
\end{equation}
for any $\vx\in \real^d$ and $\vu \in \complex^{2d}$. From this bound it
is clear that to achieve the smallest possible acoustic radiation
potential at a particular position $\vx_0$ we simply need to choose
transducer parameters $\vu_{\min}$ within the eigenspace of $\mQ(\vx_0)$
corresponding to $\lambda_{\min}(\vzero) = \lambda_{\min}(\vx_0)$. Hence $\vu_{\min}$ gives an explicit
solution to the minimization
\begin{equation}
 \min_{|\vu|^2 = 1} \psi(\vx_0;\vu),
 \label{eq:min}
\end{equation}
as was remarked in \cite{Greenhall:2015:UDS, Prisbrey:2017:UDS}. Since
the power to generate the plane waves \eqref{eq:amplitudes} is
proportional to $|\vu|^2$, the constraint in \eqref{eq:min} means that
we look only for transducer parameters that require the same power.
The bound \eqref{eq:bound} guarantees that with transducer parameters
$\vu_{\min}$, $\vx_0$ is a global minimum of $\psi(\vx;\vu_{\min})$, as
a function of $\vx$. Notice that this choice does not rule out the
existence of local minima of $\psi(\vx;\vu_{\min})$, where the particles
could also be trapped. Also because $\psi(\vx;\vu_{\min})$ is
$\mA-$periodic in $\vx$, we immediately get that $\psi(\vx;\vu_{\min})$
has global minima at points $\vx$ in the lattice
\begin{equation}
 \{ \vx_0 + \mA \vn ~|~ \vn \in \znat^d \}.
\end{equation}
A natural question is whether the global minima are limited to this
lattice.  This is answered negatively in the next section.

%%%%%%%%%%%%%%%%%%%%%%%%%%%%%%%%%%%%%%%%%%%%%%%%%%%%%%%%%%%%%%%%%%%%%%%%
\subsection{Level-sets of the acoustic radiation potential}
\label{sec:values}\
The set of all positions $\vx$ for which the acoustic radiation
potential has the same value $\gamma$ is
\begin{equation}
  L_{\gamma,\vu} = \{\vx \in \real^d ~|~ \psi(\vx;\vu) =
  \gamma\},
  \label{eq:llu}
\end{equation}
for particular transducer parameters $\vu$. To simplify the
exposition we restrict ourselves to the case $|\vu| = 1$, i.e. constant
power. From the bound \eqref{eq:bound} we see  that taking $\vu$ to be
an eigenvector of $\mQ(\vzero)$ associated with
$\lambda_{\min}(\mQ(\vzero))$ guarantees that the acoustic radiation
potential has a global minimum at the origin. In fact 
the level-sets associated with any eigenpair of
$\mQ(\vzero)$ are determined by the associated eigenspace.
\begin{lemma}
\label{lem:rayleigh}
 Let $\lambda,\vu$ be an eigenpair of $\mQ(\vzero)$ with $|\vu|=1$. Then 
 \[
  L_{\lambda,\vu} = \{ \vx \in \real^d ~|~ \exp[i [\mK,-\mK]^T \vx]
  \odot \vu \in \lambda-\text{eigenspace of}~\mQ(\vzero) \}.
 \]
\end{lemma}
\begin{proof}
Let $\vx$ such that $\exp[i[\mK,-\mK]^T\vx] \odot \vu \in \lambda-$eigenspace of
$\mQ(\vzero)$. Then from \eqref{eq:space-phase} and \eqref{eq:arpquad}, we see that
$\psi(\vx;\vu) = \psi(\vzero;\exp [i [\mK,-\mK]^T\vx] \odot \vu) = \lambda$. On the
other hand if $\vx \in L_{\lambda,\vu}$ then $\exp[i[\mK,-\mK]^T\vx] \odot \vu$
must  be a $\lambda-$eigenvector of $\mQ(\vzero)$. This is because
the Rayleigh quotient of a Hermitian matrix is equal to an eigenvalue if
and only if it is evaluated at one of the corresponding eigenvectors.
\end{proof}
Notice that by \cref{lem:Q0}, we may always be able to pick a
$\lambda-$eigenvector of $\mQ(\vzero)$ that is of the form $\vu = [\vv;
\pm\vv]$, where $|\vv| = 1/\sqrt{2}$. We now give conditions on the
symmetries of the $\lambda-$eigenspace that allow us to decide whether the
periodic patterns consist of points, lines or planes. The conditions
boil down to checking whether the $\lambda-$eigenvector $\vu$ remains a
$\lambda-$eigenvector after certain sign changes. The results are
summarized in the following theorem. We emphasize that the results in
\cref{thm:summary} do not only apply to the global minimum level-sets of
the acoustic radiation potential but also to the level-sets
corresponding to the other eigenvalues of $\mQ(\vzero)$. Nevertheless
our results do not say anything about local minima different from the
global ones.
\begin{theorem}
\label{thm:summary}
 Let $\lambda,\vu$ be an eigenpair of $\mQ(\vzero)$ with $\vu$ real and $|\vu|=1$. Then
 the acoustic radiation potential level-set $L_{\lambda,\vu}$ satisfies the following.
 \begin{itemize}
  \item If any of the entries of $\vu$ is zero, then the $\lambda-$levelset of the
  acoustic radiation potential contains lines or even planes
  (\cref{lem:indet}). This corresponds to turning off one or more of the
  transducers.

  \item If $\vu$ has no zero
  entries and the $\lambda-$eigenspace of $\mQ(\vzero)$ is all contained
  within either $H_+$ or $H_-$ then the $\lambda-$levelset of the acoustic
  radiation potential is composed of between 2 and $2^d$ isolated 
  points per primitive cell (\cref{lem:multiple}). The precise number of points
  is equal to the number of sign changes of $\vu$ for which the
  resulting vectors remain in the $\lambda-$eigenspace.

  \item If the $\lambda-$eigenspace of $\mQ(\vzero)$ straddles $H_+$ and
  $H_-$ then, the $\lambda-$levelset of the acoustic radiation potential
  may contain lines (\cref{lem:lines}) or even planes
  (\cref{lem:planes}).

 \end{itemize}
\end{theorem}
The first result applies to the situation where one or more of the
transducers is off, i.e. when we pick eigenvectors $\vu$ that have zero
entries in the $\lambda-$eigenspace of $\mQ(\vzero)$. In this case, the
level-set $L_{\lambda,\vu}$ contains subspaces of the lattice vectors,
which could be either lines or planes, depending on the dimension $d$.
\begin{lemma}
 \label{lem:indet}
 Let $\vu = [\vv;\pm \vv]$ be a real unit norm eigenvector of $\mQ(\vzero)$ associated with eigenvalue
 $\lambda$. Let $Z = \{ j ~|~ v_j =0 \}$. If any entry of $\vv$ is
 zero, i.e. $Z \neq \emptyset$ then
 \begin{equation}
  \label{eq:span}
 \{ \vk_j ~|~ j \notin Z \}^\perp = \linspan\{ \va_j ~|~ j \in Z \}
 \subset L_{\lambda,\vu}.
 \end{equation}
\end{lemma}
\begin{proof}
Let $\vz \in \{ \vk_j ~|~ j \notin Z \}^\perp$, then for $j \notin Z$, $\vk_j \cdot \vz =
0$. Or in other words for $j \notin Z$ we have  $(\exp[ i \mK^T \vz ])_j
= 1$. Hence we must have that $\exp[ i \mK^T \vz ] \odot \vv = \vv$
since
\[
 (\exp[ i \mK^T \vz ])_j v_j = 
\begin{cases}
 0   &\text{for}~j\in Z ~\text{and}\\
 v_j &\text{for}~j \notin Z.
\end{cases}
\]
We conclude that  $\exp[ i [\mK,-\mK]^T \vz ] \odot [\vv; \pm
\vv] = \vu$ is still an eigenvector of $\mQ(\vzero)$ associated with
$\lambda$ and $\psi(\vz;\vu) =
\psi(\vzero;\vu) = \lambda$ by \cref{lem:rayleigh}. The expression in terms of the $\va_j$
follows from \eqref{eq:dual}.
\end{proof}

When all the entries of the eigenvector $\vu$ in the
$\lambda-$eigenspace of $Q(\vzero)$ are non-zero (i.e. all the
transducers are activated), we can guarantee that the
level-set $L_{\lambda,\vu}$ is reduced to up to $2^d$ points per
primitive cell, where the points are determined
by whether upon changing the signs of the entries of $\vu$, the new vector
remains in the $\lambda-$eigenspace of $Q(\vzero)$.
\begin{lemma}
 \label{lem:multiple}
 Let $\vu = [\vv; \pm\vv] \in H_\pm$ be a real unit norm eigenvector of
 $\mQ(\vzero)$ associated with eigenvalue $\lambda$, such that $v_j \neq
 0$  for all $j=1,\ldots,d$. Assume that the eigenspace associated with
 $\lambda$ is a subset of $H_\pm$. Consider the set\footnote{Obviously,
 the set $T_{\lambda,\vu}$ is never empty since $\{ \vzero, \vone \}
 \subset T_{\lambda,\vu}$, because the vectors $\pm \vu$ always belong
 to the same subspace.}
 \begin{equation}
 T_{\lambda,\vu} = \{ \vs \in \{0,1\}^d ~|~
 [(-1)^{\vs} \odot \vv;\pm(-1)^{\vs} \odot\vv] \in
 \lambda\!\!-\!\!\text{eigenspace of}~\mQ(\vzero)\},
 \label{eq:sett}
 \end{equation}
 where $(-1)^{\vs} \equiv [(-1)^{s_1},\ldots,(-1)^{s_d}]$.  Then
 $L_{\lambda,\vu}$ is a union of lattices given by
 \begin{equation}
  \label{eq:slattice}
  L_{\lambda,\vu} = \bigcup_{\vs \in T_{\lambda,\vu}} 
  \{ \vx(\vn;\vs) = \mA (\vn + \vs/2) ~|~
  \vn \in \znat^d \}.
 \end{equation}
\end{lemma}
\begin{proof}
We would like to show that $\psi(\vx;\vu) = \psi(\vzero;\vu) = \lambda$
if and only if $\vx$ belongs to one of the lattices in \eqref{eq:slattice}.

Let us first assume that $\vx$ belongs to one of the lattices
in \eqref{eq:slattice}.  Then we can find $\vs \in \{0,1\}^d$ such that
$[(-1)^{\vs}\vv; \pm (-1)^{\vs}\vv]$ is in the $\lambda-$eigenspace of
$\mQ(\vzero)$ and $\vx = \mA (\vn + \vs/2)$ for some $\vn \in \znat^d$.
By \eqref{eq:dual} we have that $\exp[i\vk_j\cdot\vx] =
\exp[i2\pi(n_j+s_j/2)] = (-1)^{s_j}$ and so we also have 
\begin{equation}
 \exp[ i[\mK,-\mK]^T \vx] = \exp[ - i[\mK,-\mK]^T \vx]  =
[(-1)^{\vs}; (-1)^{\vs}].
 \label{eq:realexp}
\end{equation}
We conclude that the vector $\exp[i[\mK,-\mK]^T \vx] \odot [\vv; \pm
\vv]$ belongs to the $\lambda-$eigenspace of $\mQ(\vzero)$. In other
words, $\vx \in L_{\lambda,\vu}$ since by \eqref{eq:space-phase} we have
\[
\psi(\vx;\vu) = \psi(\vzero;\exp[i[\mK,-\mK]^T \vx] \odot [\vv; \pm \vv]) =
\psi(\vzero;\vu) = \lambda.
\]

Now let us assume that $\vx \in L_{\vu,\lambda}$, i.e. $\psi(\vx;\vu) = \lambda$. By
using \eqref{eq:space-phase} again we have that $\exp[i[\mK,-\mK]^T \vx]
\odot [\vv; \pm \vv]$ must be in the $\lambda-$eigenspace of
$\mQ(\vzero)$.  Moreover the former vector can be split into components
in $H_\pm$ and $H_{\mp}$ as follows
\begin{equation}
\exp\Mb{i\begin{bmatrix}\mK^T\\-\mK^T\end{bmatrix} \vx}
\odot \begin{bmatrix}\vv\\\pm\vv\end{bmatrix} =
\underbrace{
\cos\Mb{\begin{bmatrix}\mK^T\\-\mK^T\end{bmatrix} \vx}
\odot \begin{bmatrix}\vv\\\pm\vv\end{bmatrix}}_{\in H_\pm}
+
i\underbrace{\sin\Mb{\begin{bmatrix}\mK^T\\-\mK^T\end{bmatrix} \vx}
\odot \begin{bmatrix}\vv\\\pm\vv\end{bmatrix}}_{\in H_\mp}.
\label{eq:hpmmp}
\end{equation}
Since we assumed that the $\lambda-$eigenspace of $\mQ(\vzero)$ is a
subspace of $H_{\pm}$ we must have that $\exp[i[\mK,-\mK]^T \vx]
\odot [\vv; \pm \vv] \in H_{\pm}$ and its $H_{\mp}$ component must be
zero. By the decomposition \eqref{eq:hpmmp}, this means that $\exp[i[\mK,-\mK]^T \vx]
\odot [\vv; \pm \vv]$ must be real. Using the definition \eqref{eq:sett}
of the set $T_{\lambda,\vu}$, we see that there must be an $\vs \in \{0,1\}^d$ such that
\eqref{eq:realexp} holds. This imposes that $\vx$ must be part of the
lattice \eqref{eq:slattice} with $\vs \in T_{\lambda,\vu}$.
\end{proof}

To better illustrate \cref{lem:multiple}, let us define the primitive cell 
\begin{equation}
C = \{ \mA \valpha ~|~ \valpha \in [0,1]^d\}.
\label{eq:primitive}
\end{equation} 
Points within the primitive cell can be identified by their ``atomic
coordinates'' $\valpha \in [0,1]^d$.  Thus we have the following extreme
cases.
 \begin{enumerate}[i.]
  \item If $\lambda$ is a simple eigenvalue then $T_{\lambda,\vu} = \{
  \vzero, \vone \}$ and there are two points per primitive cell in
  $L_{\lambda,\vu}$, namely $\vzero$ and $\vone/2$, in atomic coordinates,
  provided $\vu$ is a unit length eigenvector of $\mQ(\vzero)$ with
  non-zero entries.

  \item If $\lambda$ is an eigenvalue of multiplicity $d$, then
  $T_{\lambda,\vu} = \{ 0,1\}^d$ and there are exactly
  $2^d$ points in $L_{\lambda,\vu}$ per primitive cell, namely the points
  with atomic coordinates $\vs/2$, where $\vs \in \{0,1\}^d$. This is of
  course provided $\vu$ is a unit length eigenvector of $\mQ(\vzero)$
  with non-zero entries and the $\lambda-$eigenspace of $\mQ(\vzero)$ is
  all within either $H_+$ or $H_-$.

 \end{enumerate}
 
\begin{example}[Eigenvalue of multiplicity 2]
\label{ex:em2}

Consider a 2D example. Choose the wavevectors so that $\mK=\mI_2$ and
$\ell = 2\pi$. Setting
$\mathfrak{a}=\mathfrak{b}=1$ in the acoustic radiation potential and using \eqref{eq:Q0} we
get
\begin{equation}
\begin{aligned}
\mQ(\vzero) &= \vone\vone^T 
-\begin{bmatrix}\mK\\-\mK\end{bmatrix}
 \begin{bmatrix}\mK\\-\mK\end{bmatrix}^T
 = \begin{bmatrix}
 0 & 1 & 2 & 1 \\
 1 & 0 & 1 & 2 \\
 2 & 1 & 0 & 1 \\
 1 & 2 & 1 & 0
 \end{bmatrix}.
\end{aligned}
\end{equation}
We choose the eigendecomposition $\mQ(\vzero) = \mU \mLambda \mU^T$
with
\begin{equation}
\mU = 
\begin{bmatrix*}[r]
1/2 & -1/2 & 1/\sqrt{2} & 0 \\
1/2 & 1/2 & 0 & 1/\sqrt{2} \\
1/2 & -1/2 & -1/\sqrt{2} & 0 \\
1/2 & 1/2 & 0 & -1/\sqrt{2}
\end{bmatrix*},
\end{equation}
and $\diag(\mLambda) = \{4,0,-2,-2\}$ such that it conforms to $H_+$ and
$H_-$, as in \cref{lem:Q0}.
Notice that the $(-2)-$eigenspace of $\mQ(\vzero)$ is identical
to $H_-$. Pick a real unit length eigenvector $\vu$ with no zero entries
in this eigenspace. Clearly we must have $T_{-2,\vu} = \{0,1\}^2$ (see
\cref{lem:multiple} for the definition of this set). Hence
\cref{lem:multiple}, predicts acoustic radiation potential minima at the
union of lattices 
\begin{equation}
 \label{eq:exlattice}
  \bigcup_{ \vs \in \{0,1\}^2}
  \{\vx(\vn;\vs) = 2\pi(\vn+\vs/2) ~|~ \vn\in \znat^d\}.
\end{equation}
\Cref{fig:multiplicity} shows that the acoustic
radiation potential for this example with eigenvector $\vu =
[1,1,-1,-1]^T/2$ has 4 minimum points per primitive cell.

If we consider the acoustic radiation potential for an eigenvector with
zero entries, we no longer have strict minima, by \cref{lem:indet}. This
is illustrated in \cref{fig:tetragonal_planes}, which shows the acoustic
radiation potential for the same example, but with eigenvector $\vu =
[1,0,-1,0]^T/\sqrt{2}$.
\end{example}

\begin{figure}
\begin{center}
\includegraphics[scale=0.1]{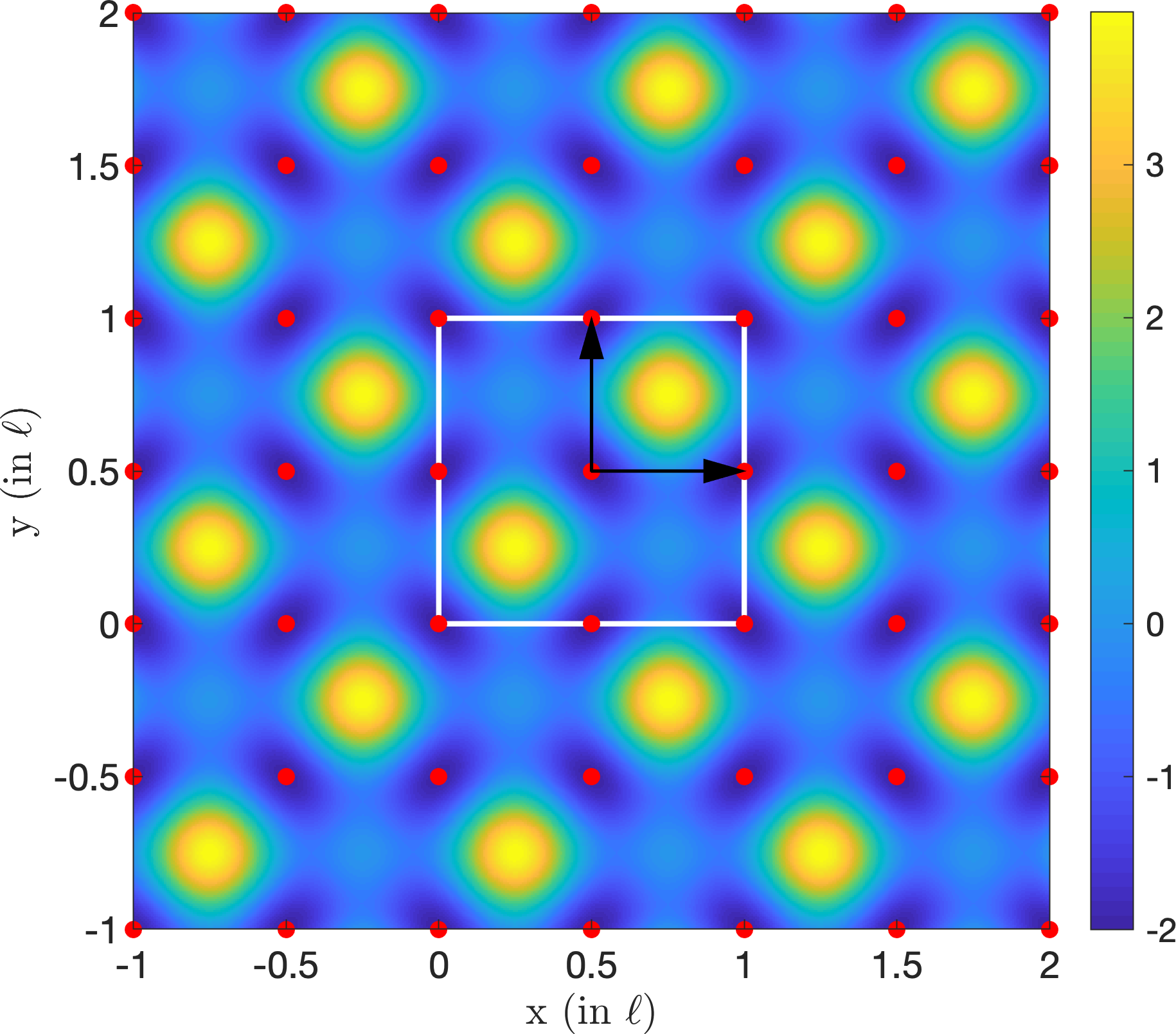}
\end{center}
\caption{Acoustic radiation potential (from \cref{ex:em2}) resulting in a tetragonal lattice
arrangement of minima when the eigenvector has no zero entries. A
primitive cell is outlined in white. The points in the lattice
\eqref{eq:exlattice} are shown in red. Since the minimum eigenvalue of
$\mQ(\vzero)$ has
multiplicity 2, there are 4 minimum points per primitive
cell. The black arrows indicate the directions normal to the
transducers.}
\label{fig:multiplicity}
\end{figure}

\begin{figure}
\begin{center}
\includegraphics[scale=0.1]{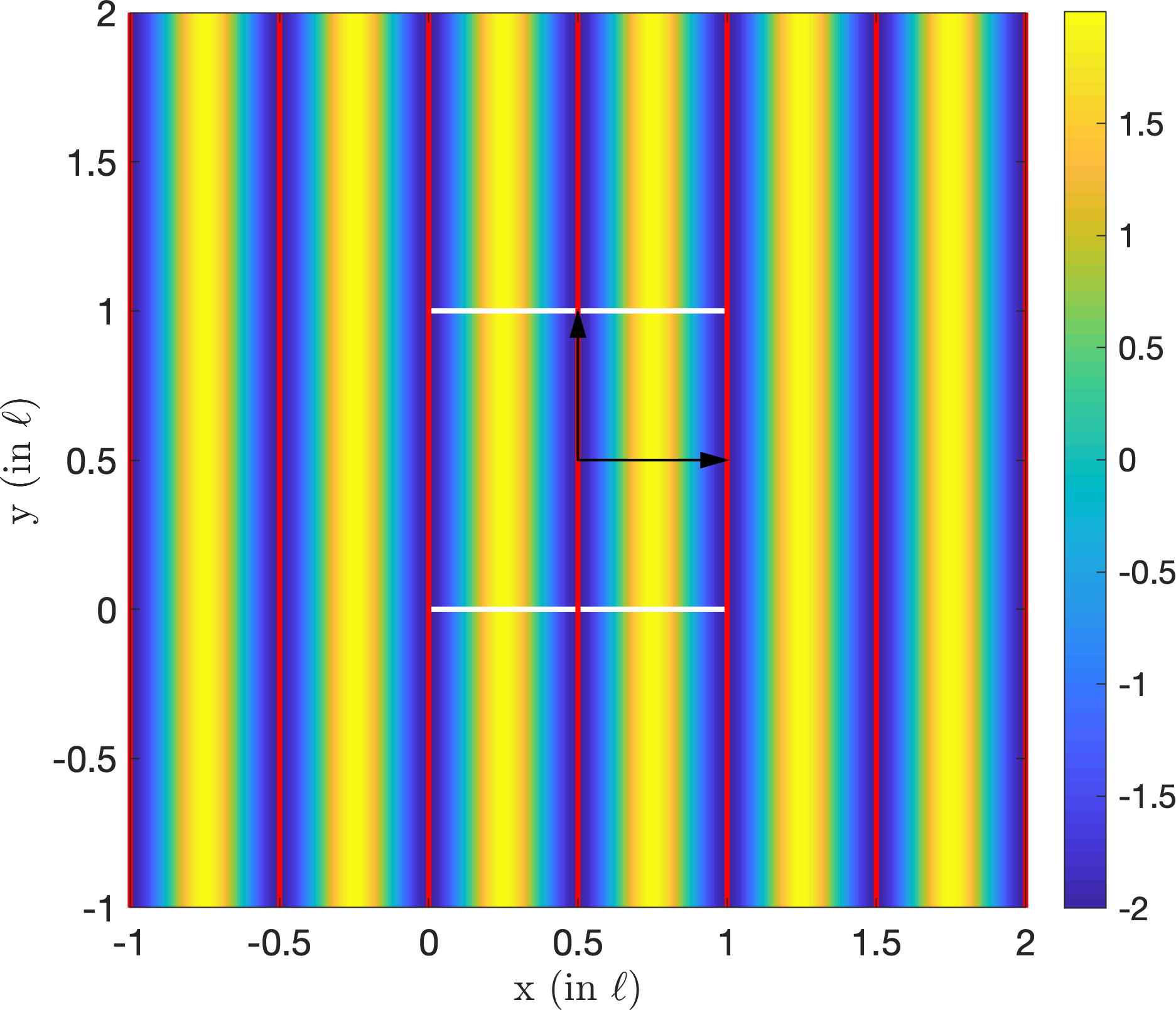}
\end{center}
\caption{The acoustic radiation potential defined in \cref{ex:em2}
results in lines of minima if the eigenvector used to compute the
acoustic radiation potential has
zero entries. The lines where minima lie are the spans of the lattice
vectors specified by \eqref{eq:span} and are indicated in red. A
primitive cell is outlined in white. The black arrows indicate the directions normal to the
transducers.}
\label{fig:tetragonal_planes}
\end{figure}

 \begin{lemma}
 \label{lem:lines}
 Let $\vu = [\vv; \pm \vv]\in H_{\pm}$ be a real unit norm eigenvector of
 $\mQ(\vzero)$ associated with eigenvalue $\lambda$.
 Consider the set $T^\pm_{\lambda,\vu}$ defined by
 \begin{equation}
  T^\pm_{\lambda,\vu} = \left\{ \vs \in \{0,1\}^d ~|~ [ (-1)^{\vs} \odot \vv ;
  \mp (-1)^{\vs}
  \odot \vv] \in \lambda\!\!-\!\!\text{eigenspace of}~\mQ(\vzero)
  \right\}.
  \label{eq:newt}
 \end{equation}
 Notice that if the set $T^\pm_{\lambda,\vu}$ is not empty, the
 $\lambda-$eigenspace of $\mQ(\vzero)$ straddles between $H_+$ and
 $H_-$, which is a possibility predicted by \cref{lem:Q0}.
 Then we can guarantee that $L_{\lambda,\vu}$ contains the lines
 \begin{equation}
  \label{eq:sline}
  \{ \vx(\vn;\theta) = \mK^{-T}(\theta (-1)^{\vs} + 2\pi\vn) ~|~ \theta \in
  \real\}~\text{for}~ \vn \in \znat^d~\text{and}~\vs \in
  T^\pm_{\lambda,\vu}.
 \end{equation}
 \end{lemma}
 \begin{proof}
 Assume that $\vx$ belongs to one of the lines \eqref{eq:sline}, we
 would like to show that $\psi(\vx;\vu) = \psi(\vzero;\vu) = \lambda$.
 Then $\mK^T\vx = \theta (-1)^{\vs} + 2\pi\vn$, and we have that
 $\exp[i\vk_j\cdot \vx] = \exp[i\theta (-1)^{s_j}]$. 
 The following decomposition in terms of a vector in $H_\pm$ and in
 $H_\mp$ holds
 \begin{equation}
\exp\left[i\begin{bmatrix}
\mK^T \\
-\mK^T
\end{bmatrix}\vx\right]
\odot
\begin{bmatrix}
\vv \\
\pm \vv
\end{bmatrix}
= \cos\theta
\begin{bmatrix}
\vv \\
\pm \vv
\end{bmatrix}
+ i\sin\theta
\begin{bmatrix}
(-1)^{\vs} \\
(-1)^{\vs}
\end{bmatrix}
\odot
\begin{bmatrix}
\vv \\
\mp \vv
\end{bmatrix}.
 \end{equation}
By the definition of the set $T^\pm_{\lambda,\vu}$, the vector $[(-1)^{\vs} \odot \vv;
\mp (-1)^{\vs} \odot \vv]$ must be a $\lambda-$eigenvector of
$\mQ(\vzero)$. Hence $\exp[i[\mK^T; -\mK^T]\vx]\odot [\vv; \pm\vv]$ is
also a $\lambda-$eigenvector of $\mQ(\vzero)$. Again by
\eqref{eq:space-phase}, we readily get that $\psi(\vx;\vu) = \lambda$.
 \end{proof}
 
 \begin{example}[Lines of minima]
 \label{ex:lom}
 Consider a 2D example. Choose the wavevectors so $\mK=\mI_2$,
 $\ell=2\pi$ and set
 $\mathfrak{a}=1, \mathfrak{b}=0$ in the acoustic radiation potential. By \eqref{eq:Q0}, we
 have
 $\mQ(\vzero) = \vone\vone^T$. We choose the eigendecomposition $\mQ(\vzero)
 = \mU \mLambda \mU^T$ with $\diag(\mLambda) =  \{4,
 0, 0, 0\}$ and 
 \begin{equation}
 \mU = \frac{1}{\sqrt{2}}
 \begin{bmatrix*}[r]
 1/\sqrt{2} & -1/\sqrt{2} & 1 & 0 \\
 1/\sqrt{2} & 1/\sqrt{2} & 0 & 1 \\
 1/\sqrt{2} & -1/\sqrt{2} & -1 & 0 \\
 1/\sqrt{2} & 1/\sqrt{2} & 0 & -1
 \end{bmatrix*}.
 \end{equation}
 This eigendecomposition conforms with $H_+$ and $H_-$ and is consistent
 with \cref{lem:Q0}.  Notice that the $0-$eigenspace of $\mQ(\vzero)$
 straddles over $H_+$ and $H_-$, and that this eigenspace contains all
 of $H_-$. The real unit length vector $\vu = [-1,1,-1,1]^T/2$ (the second
 column of $\mU$) belongs to both $H_+$ and the $0-$eigenspace of
 $\mQ(\vzero)$. Using definition \eqref{eq:newt} we can verify that
 $T^\pm_{0,\vu} = \{0,1\}^2$. Thus \cref{lem:lines} predicts
 acoustic radiation potential minima along the families of lines
\begin{equation}
\label{eq:exline}
\{ \vx(\vn;\theta) = \mK^{-T}(\theta (-1)^{\vs} + 2\pi\vn) ~|~ \theta \in
\real \},~\text{for}~\vn \in \znat^d~\text{and}~\vs\in\{0,1\}^2.
\end{equation}
These lines are indicated in \cref{fig:lines} and
coincide with minima of the acoustic radiation potential.
\end{example}

\begin{figure}
\begin{center}
\includegraphics[scale=0.1]{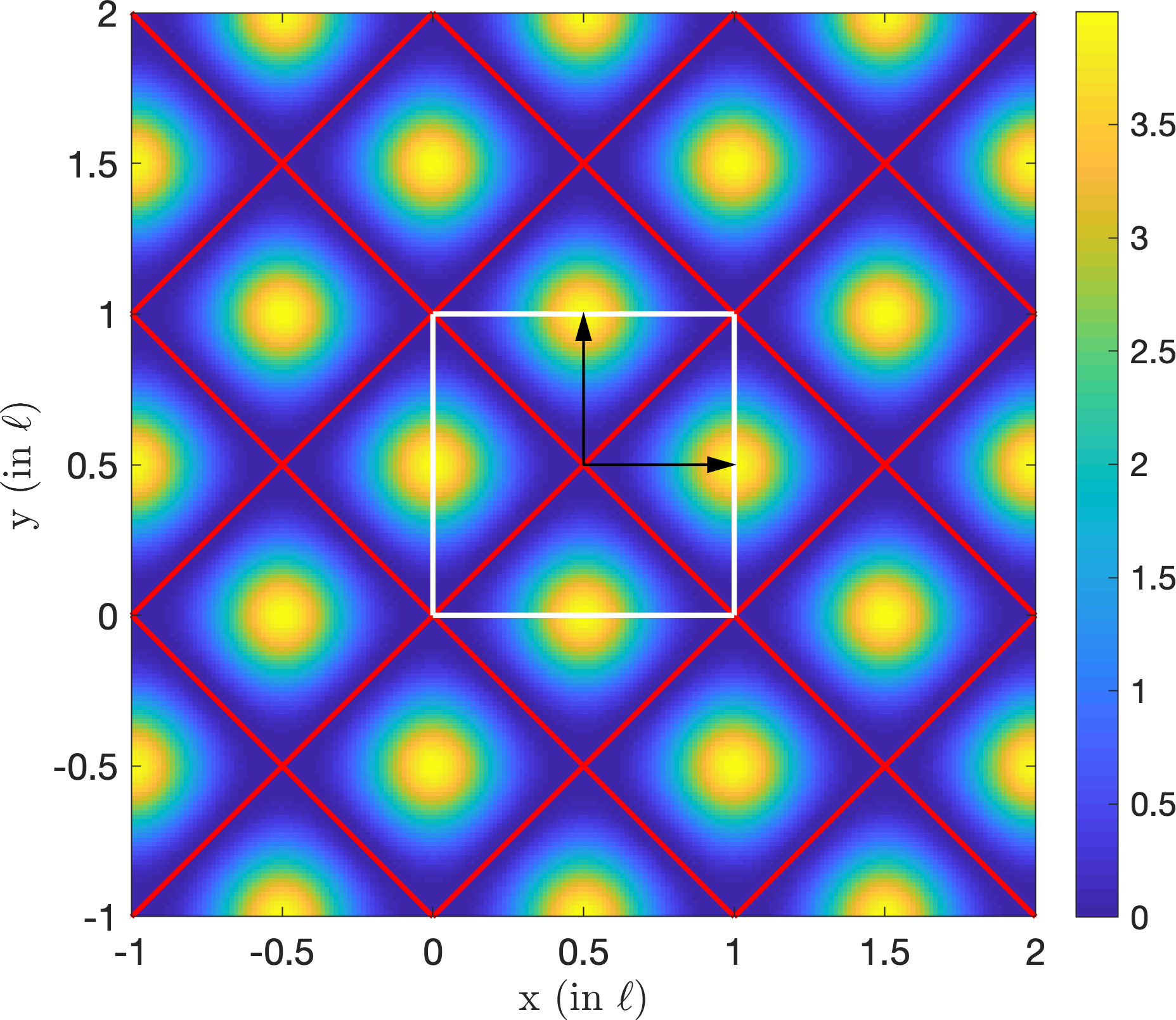}
\end{center}
\caption{Acoustic radiation potential of \cref{ex:lom} resulting in
lines of minima. A primitive cell is outlined in white. The lines
predicted by \eqref{eq:exline} are shown in red. The black arrows
indicate the directions normal to the transducers.} \label{fig:lines}
\end{figure}

\begin{lemma}
\label{lem:planes}
Let $\vu = [\vv; \pm \vv] \in H_{\pm}$ be a real unit norm eigenvector of $\mQ(\vzero)$ associated with eigenvalue $\lambda$. Consider the set $R_{\lambda,\vu}^{\pm}$ defined by
\begin{equation}
\begin{aligned}
R_{\lambda,\vu}^{\pm} = \bigg\{(\vs, \vr) \in (\{0,1\}^d)^2 ~\bigg|~ 
&
\begin{bmatrix}
(-1)^{\vs} \\
(-1)^{\vs}
\end{bmatrix}
\odot
\begin{bmatrix}
\vv \\
\mp\vv
\end{bmatrix},
\begin{bmatrix}
(-1)^{\vr} \\
(-1)^{\vr}
\end{bmatrix}
\odot
\begin{bmatrix}
\vv \\
\mp \vv
\end{bmatrix},
 \\
&\begin{bmatrix}
(-1)^{\vs+\vr} \\
(-1)^{\vs+\vr}
\end{bmatrix}
\odot
\begin{bmatrix}
\vv \\
\pm\vv
\end{bmatrix}
\in \lambda\!\!-\!\!\text{eigenspace of}~\mQ(\vzero)
%, \text{and}~\vs+\vr \neq \vone
\bigg\}.
\end{aligned}
\label{eq:setr}
\end{equation}
Notice that if the set $R_{\lambda,\vu}^{\pm}$ is not empty, the
$\lambda$-eigenspace of $\mQ(\vzero)$ straddles between $H_+$ and $H_-$.
Then we can guarantee $L_{\lambda,\vu}$ contains the sets
\begin{equation}
\{\vx(\vn;\theta,\phi) = \mK^{-T}(\theta (-1)^{\vs}+\phi (-1)^{\vr} +
2\pi\vn) ~|~ \theta, \phi\in \real\},
\label{eq:planes}
\end{equation}
for $\vn\in \znat^d$ and $(\vs,\vr)\in R^\pm_{\lambda,\vu}$. The sets in
\eqref{eq:planes} are guaranteed to be planes when the vectors
$(-1)^{\vr}$ and $(-1)^{\vs}$ are linearly independent or equivalently
$(-1)^{\vr+\vs} \notin \{-\vone,\vone\}$.
\end{lemma}
\begin{proof}
Assume $\vx$ belongs to one of the planes \eqref{eq:planes}, we would
like to show that $\psi(\vx;\vu) = \psi(\vzero;\vu) = \lambda$ .  Since
$\mK^T\vx = \theta (-1)^{\vs} + \phi (-1)^{\vr} + 2\pi \vn$, we must have
that $\exp[i\vk_j\cdot \vx] = \exp[i(\theta (-1)^{s_j} + \phi
(-1)^{r_j})]$. Then we obtain the following decomposition in terms of
vectors in $H_{\pm}$ and $H_{\mp}$,
\begin{equation}
\begin{aligned}
\exp\left[i
\begin{bmatrix}
\mK^T \\
-\mK^T
\end{bmatrix}\vx\right]
\odot
\begin{bmatrix}
\vv \\ 
\pm\vv
\end{bmatrix}
= &\left(
\cos\theta\cos\phi
\begin{bmatrix}
 \vone\\ \vone
\end{bmatrix}
-\sin\theta\sin\phi
\begin{bmatrix}
(-1)^{\vs+\vr} \\
(-1)^{\vs+\vr}
\end{bmatrix}
\right)
\odot
\begin{bmatrix}
\vv \\
\pm \vv
\end{bmatrix} \\
 + i&\left(
 \sin\theta\cos\phi
 \begin{bmatrix}
 (-1)^{\vs} \\
 (-1)^{\vs}
 \end{bmatrix}
 + \cos\theta\sin\phi
 \begin{bmatrix}
 (-1)^{\vr} \\
 (-1)^{\vr}
 \end{bmatrix}
 \right)
 \odot
 \begin{bmatrix}
 \vv \\
 \mp \vv
 \end{bmatrix}.
\end{aligned}
\end{equation}
By definition of the set $R_{\lambda,\vu}^{\pm}$, each vector in the
decomposition must be a $\lambda-$eigenvector of $\mQ(\vzero)$. Hence
$\exp[i[\mK^T;-\mK^T]\vx]\odot [\vv; \pm\vv]$ is also a
$\lambda$-eigenvector of $\mQ(\vzero)$. By \eqref{eq:space-phase} this
implies that $\psi(\vx;\vu) = \lambda$. 
\end{proof}

\begin{example}[Planes of minima]
\label{ex:planes}
Consider a 3D example. Choose the wavevectors so $\mK = \mI_3$. Setting
$\mathfrak{a}=1, \mathfrak{b}=0$ in the acoustic radiation potential gives $\mQ(\vzero) =
\vone\vone^T$. An eigendecomposition $\mQ(\vzero) = \mU \mLambda \mU^T$ is
chosen such that  $\diag(\mLambda) = \{6,0,0,0,0,0\}$ and
\begin{equation}
\mU = \frac{1}{\sqrt{2}}
\begin{bmatrix*}[r]
1/\sqrt{3} & 1/\sqrt{2} & 1/\sqrt{6} & 1 & 0 & 0 \\
1/\sqrt{3} & -1/\sqrt{2} & 1/\sqrt{6} & 0 & 1 & 0 \\
1/\sqrt{3} & 0 & -2/\sqrt{6} & 0 & 0 & 1 \\
1/\sqrt{3} & 1/\sqrt{2} & 1/\sqrt{6} & -1 & 0 & 0 \\
1/\sqrt{3} & -1/\sqrt{2} & 1/\sqrt{6} & 0 & -1 & 0 \\
1/\sqrt{3} & 0 & -2/\sqrt{6} & 0 & 0 & -1
\end{bmatrix*},
\label{eq:qplane}
\end{equation}
conforming to $H_+$ and $H_-$ (as in \cref{lem:Q0}).
The $0-$eigenspace of $\mQ(\vzero)$ straddles over $H_+$ and $H_-$, and
contains all of $H_-$. The real unit length vector $\vu =
[1,-1,0,1,-1,0]^T/2$ (the second column of $\mU$) belongs to both $H_+$ and the $0-$eigenspace of $\mQ(\vzero)$. Using definition \eqref{eq:setr} we can verify that 
\begin{equation}
R_{0,\vu}^{\pm} = \bigg\{ (\vs,\vr) \in \{0,1\}^3 \times \{0,1\}^3 ~|~ (-1)^{s_1+r_1} =
(-1)^{s_2+r_2} \bigg\}.
\end{equation}
Thus \cref{lem:planes} predicts acoustic radiation potential minima
along the sets given in \cref{eq:planes}, some of them being the planes 
indicated in \cref{fig:arp_planes}. We have verified numerically that
they coincide with the minima of the acoustic radiation potential.
\end{example}

\begin{figure}
\begin{center}
\begin{tabular}{cc}
\includegraphics[scale=0.1]{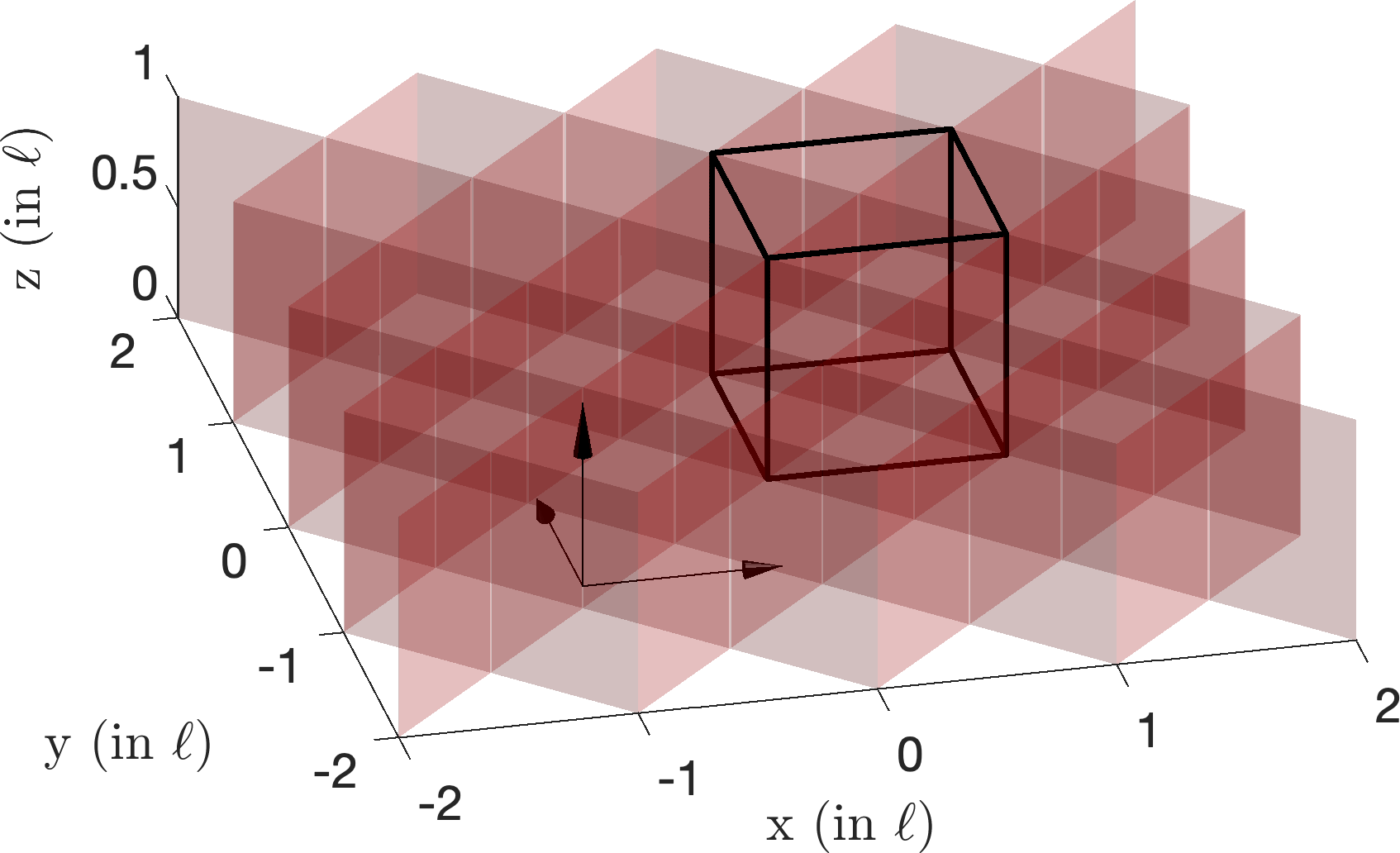}
&
\includegraphics[scale=0.1]{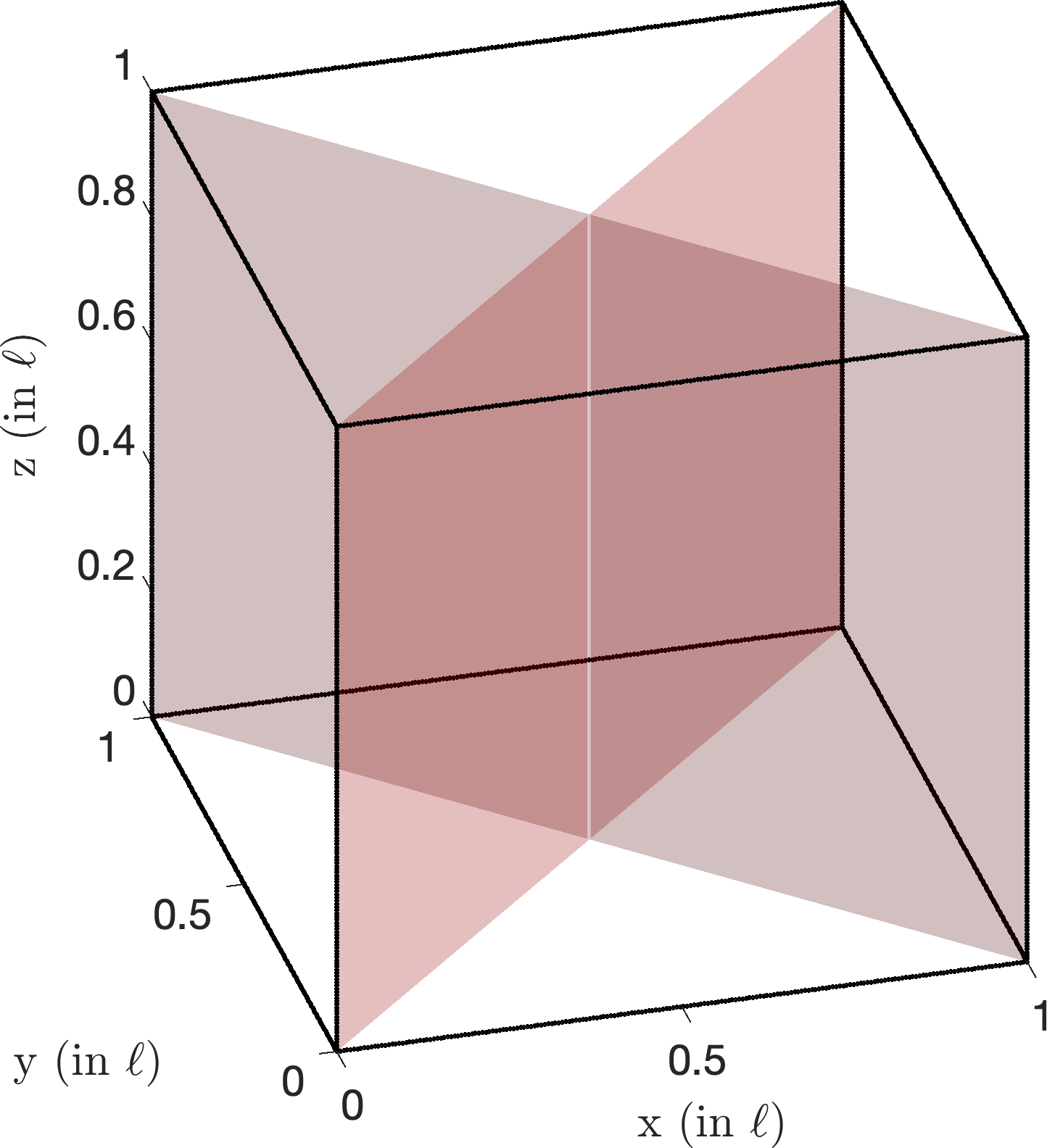}
\end{tabular}
\end{center}
\caption{The minima of the acoustic radiation potential defined in
\cref{ex:planes} appear on planes. The planes are displayed on a few
unit cells (left) and for more clarity on a unit cell (right). The black
arrows indicate the directions normal to the transducers.}
\label{fig:arp_planes}
\end{figure}

\begin{example}
\label{ex:3dlines}
Using the same matrix in \eqref{eq:qplane} but choosing the unit-length
$0-$eigenvector $\vv = [1,1,-2,1,1,-2] /2\sqrt{3} \in H_+$ we get
\begin{equation}
 R_{0,\vv}^{\pm} = \bigg\{ (\vs,\vr) \in \{0,1\}^3 \times \{0,1\}^3 ~|~
 (-1)^{s_1+r_1} = (-1)^{s_2+r_2} = (-1)^{s_3+r_3} \bigg\}.
\end{equation}
Now all the corresponding sets in \eqref{eq:planes} are lines because we
have that $(-1)^{\vs+\vr} \in \{-\vone,\vone\}$ for $(\vs,\vr)
\in R_{0,\vv}^{\pm}$. These lines are displayed in \cref{fig:3dlines}.
\end{example}

\begin{figure}
\begin{center}
\begin{tabular}{cc}
\includegraphics[scale=0.1]{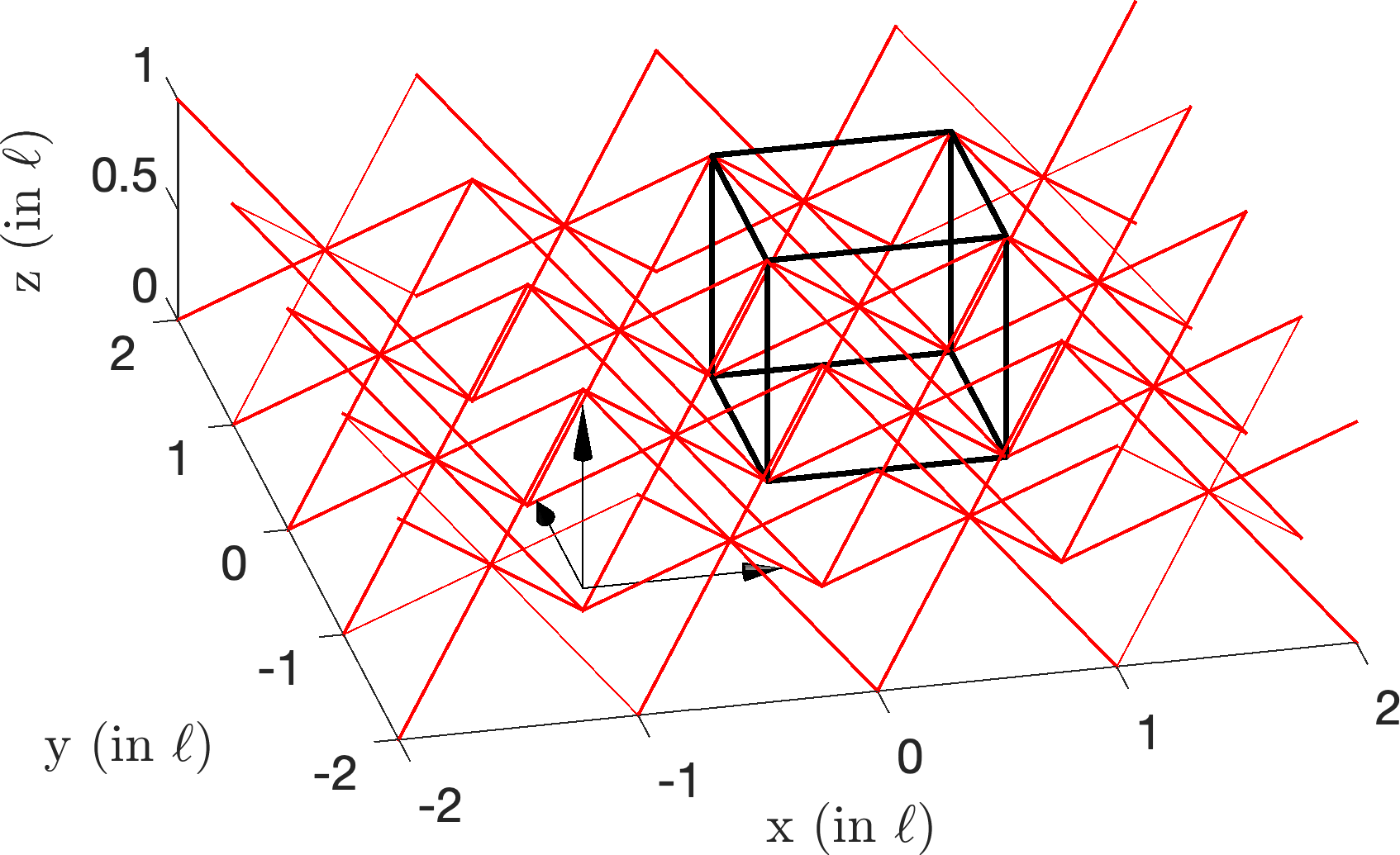}
&
\includegraphics[scale=0.1]{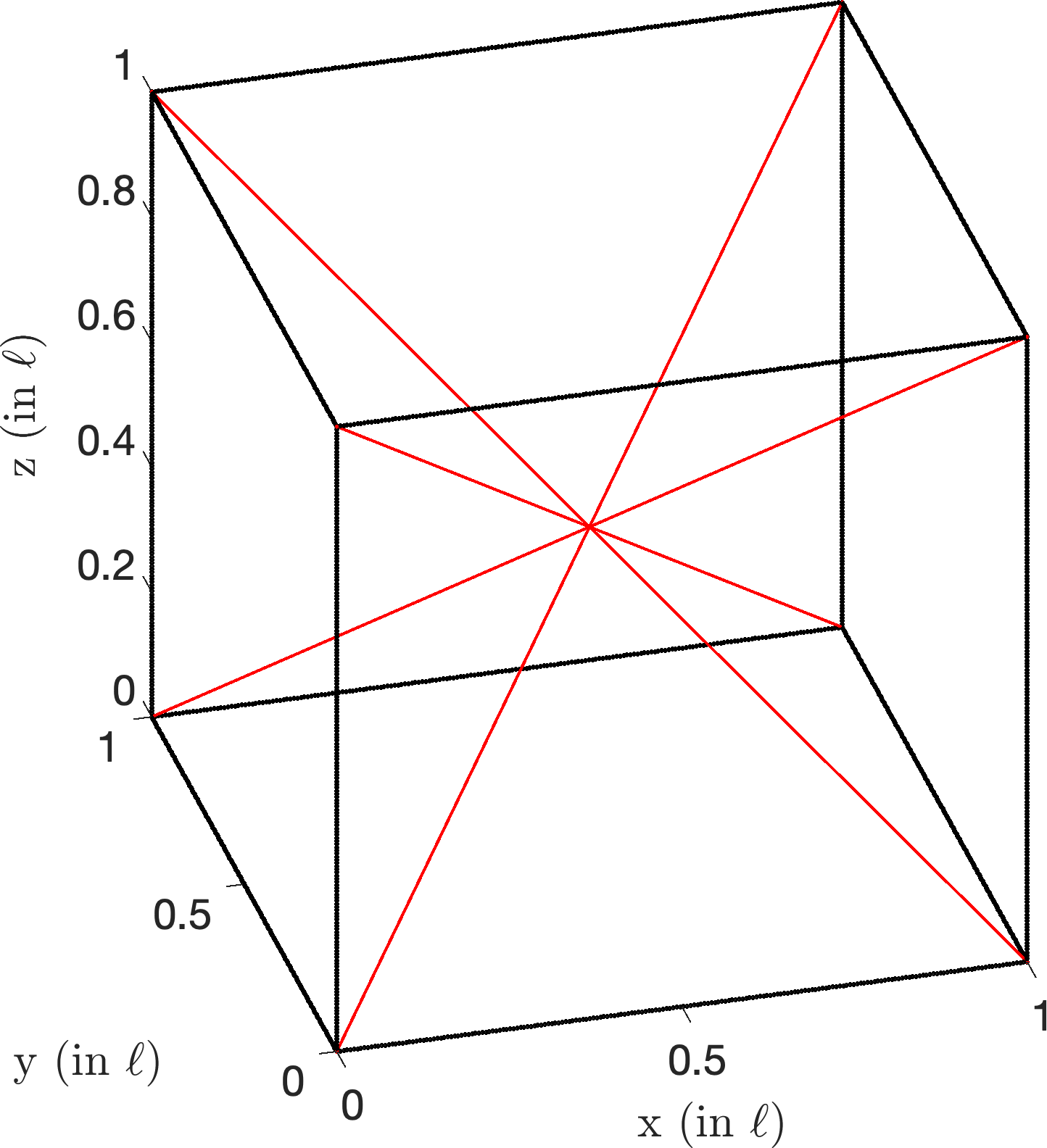}
\end{tabular}
\end{center}
\caption{The minima of the acoustic radiation potential in
\cref{ex:3dlines} appear on lines.  The lines are displayed on a few
unit cells (left) and for more clarity on a unit cell (right). The black
arrows indicate the directions normal to the transducers.}
\label{fig:3dlines}
\end{figure}

%%%%%%%%%%%%%%%%%%%%%%%%%%%%%%%%%%%%%%%%%%%%%%%%%%%%%%%%%%%%%%%%%%%%%%%%
\section{Achievable Bravais lattice classes}
\label{sec:lims}
In earlier sections, we have taken the wavevectors $\vk_1, \ldots ,
\vk_d$ to be fixed. Now, we address the question of how to choose these
vectors in order to obtain an arrangement of particles on a particular
Bravais lattice. The position of the particles is dictated by the
acoustic radiation potential $\psi(\vx;\vu)$ which is $\mA-$periodic.
Since $\mA = 2\pi \mK^{-T}$, and all the columns of $\mK$ have the same
length $k$, there are limitations to the possible Bravais lattices that we
can obtain. To explore these limitations, we introduce the following
definition.
\begin{definition}
\label{def:achievable}
We say that a Bravais lattice class is {\em achievable} if out of all class
members with reciprocal lattice vectors $\vk_1,\ldots,\vk_d$ having same
length, i.e.
\begin{equation}
 |\vk_1| = \cdots = |\vk_d|
\end{equation}
there is at least one member that does not belong to any of the other Bravais
lattice classes. 
\end{definition}
In other words, a Bravais
lattice class {\em cannot be achieved} if under the constraint of
\cref{def:achievable}, all Bravais
lattices within a class happen to belong to another class.  For a two
dimensional example, notice that tetragonal lattices are members of the
orthorhombic lattice class. However, when we place the same length constraint
on the reciprocal lattice vectors for an orthorhombic lattice, we end
up with a tetragonal lattice. Thus we say that the orthorhombic lattice
class cannot be achieved. We remark that \cref{def:achievable} is
irrespective of the particular particle arrangement inside a primitive
cell, which could include isolated points, lines or planes
(\cref{thm:summary}).

Our results are summarized in \cref{tab:bravais2d,tab:bravais3d} for two
and three dimensions, respectively. We found that out of the 5 Bravais
lattice classes  in two dimensions (see e.g.  \cite{Kittel:2005:ISP})
only three are achievable. In three dimensions, there are 14 Bravais
lattice classes (see e.g.  \cite{Kittel:2005:ISP}) but we found that
only 6 are achievable.

To generate \cref{tab:bravais2d,tab:bravais3d}, we took known reference
tables associating Bravais lattice classes to the typical form of their
reciprocal vectors. Then we imposed the condition that all the
reciprocal vectors have the same norm\footnote{For certain lattices we
used reciprocal vectors with norm different from one, in an effort to get
simpler expressions.}. A general form for the reciprocal vectors is given in the
second column of these tables. If the third column has an entry, then
the particular class cannot be achieved and the class that is implied by
the same norm reciprocal vector constraint is indicated. If the third column has
no entry, it means that at least one representative belonging
exclusively to the class can be achieved. Of course,
\cref{tab:bravais2d,tab:bravais3d} can be used to design lattices by
taking the transducer normal orientations to be those in the second
column.  The reference tables we based our study on can be found in
\cite[Table 3.3]{Bradley:2010:MTSS} for three dimensions. The two
dimensional reciprocal vectors can be easily derived from e.g.
\cite[Fig. 1.7]{Kittel:2005:ISP}. We also illustrate in
\cref{fig:bravais2d,fig:bravais3d} representatives of the classes of
Bravais lattices that are achievable using standing acoustic waves, in
the particular case of isolated particle arrangements.  

\begin{table}
\begin{center}
 \begin{tabular}{|c|c|c|}
 \hline
 Bravais lattice class & Reciprocal lattice vectors & Implied symmetry\\
 \hline
 \multirow{2}{10em}{\makecell{Monoclinic}} & $\vg_1 = (1,-\cot \gamma)$ & \\
 & $\vg_2 = (0, \csc \gamma)$ & Orthorhombic centred\\
 \hline
\multirow{2}{10em}{\makecell{Orthorhombic}} & $\vg_1 = (1,0)$ & \\
& $\vg_2 = (0,1)$ & Tetragonal\\
 \hline
 \multirow{2}{10em}{\makecell{Orthorhombic centred}} & $\vg_1 = (\csc(\gamma)\sin(\gamma/2),\csc(\gamma/2)/2)$ & \\
 & $\vg_2 = (\csc(\gamma)\sin(\gamma/2),-\csc(\gamma/2)/2)$ &\\
 \hline
\multirow{2}{10em}{\makecell{Hexagonal}} & $\vg_1 = (1,1/\sqrt{3})$ & \\
& $\vg_2 = (0,2/\sqrt{3})$ &\\
 \hline
\multirow{2}{10em}{\makecell{Tetragonal}} & $\vg_1 = (1,0)$ & \\
& $\vg_2 = (0,1)$ &\\
 \hline
 \end{tabular}
 \end{center}
 \caption{Two dimensional Bravais lattice classes that are achievable using
 standing acoustic waves. The reciprocal lattice vectors we give satisfy
 $|\vg_1| = |\vg_2|$, so they need to be rescaled so that their length is
 $k$ in order to be interpreted as the wavevectors needed to obtain a particular
 Bravais lattice. The angle $\gamma$ is the angle between the
 two primitive vectors in the lattice.}
 \label{tab:bravais2d}
\end{table}

\begin{figure}
\begin{center}
\begin{tabular}{cc}
Orthorhombic centred & Hexagonal\\
\includegraphics[scale=0.09]{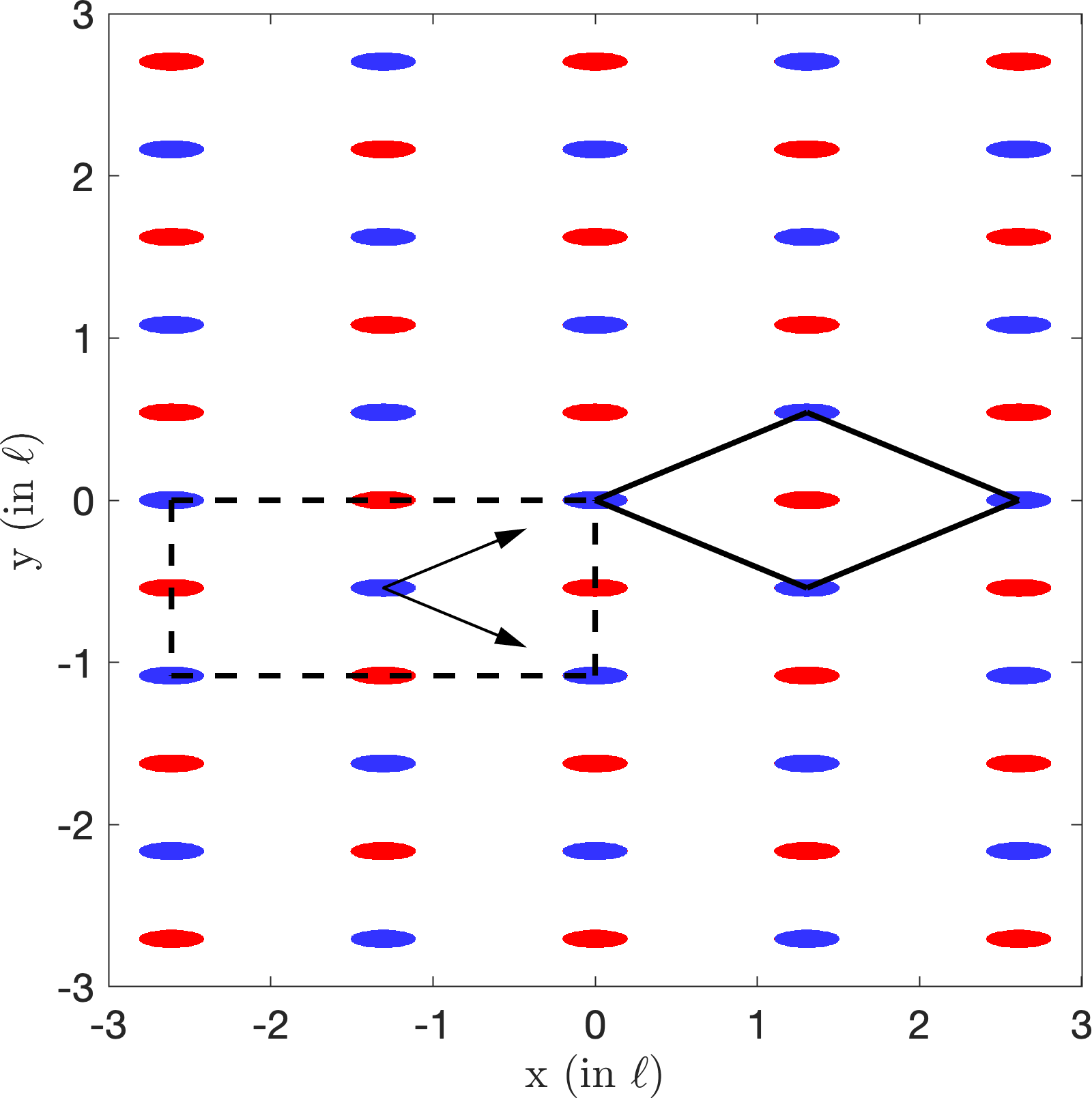} &
\includegraphics[scale=0.09]{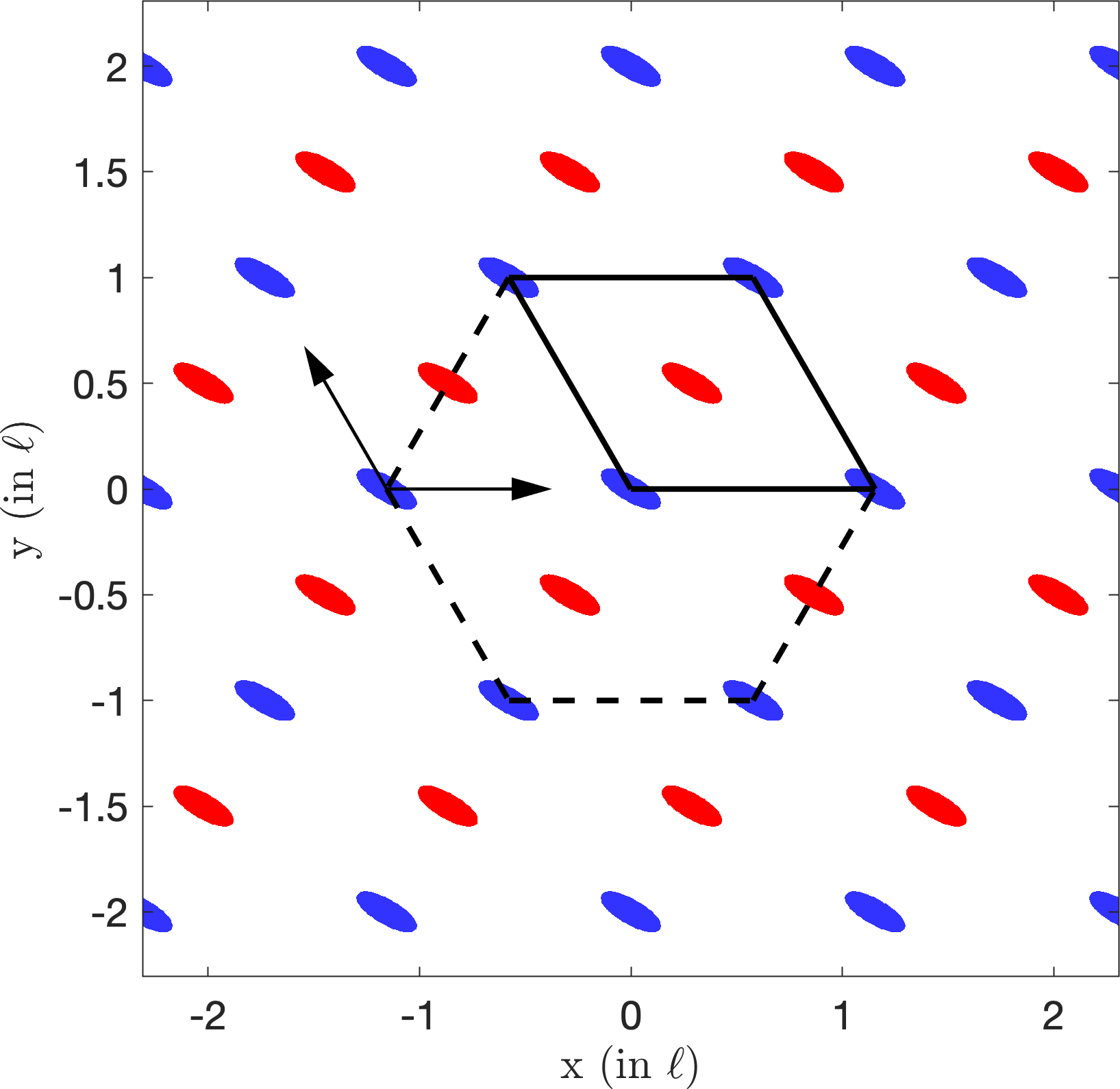}\\
\includegraphics[scale=0.09]{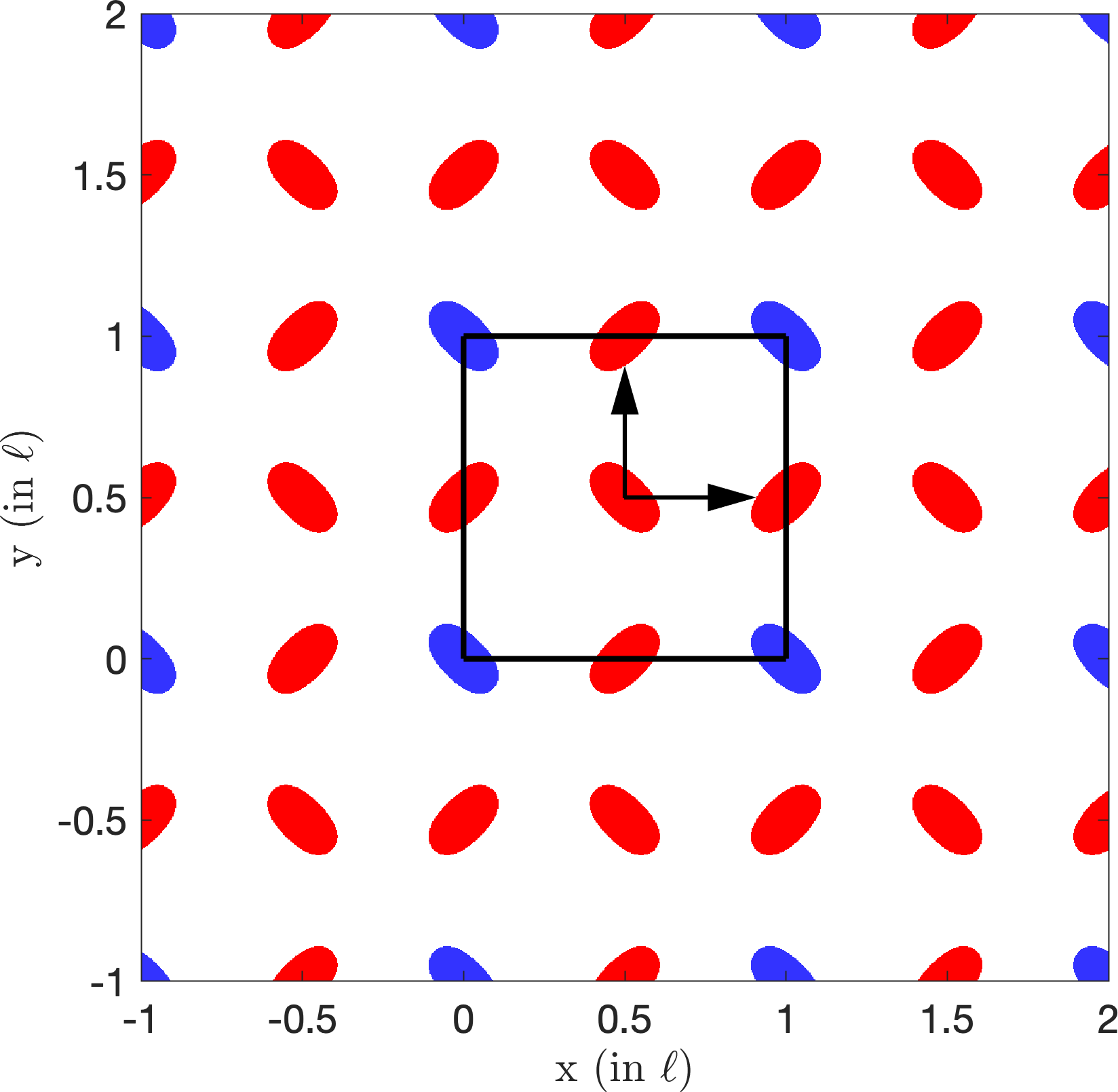} \\
Tetragonal & \\
\end{tabular}
\end{center}
\caption{Representatives of the three Bravais lattice classes that
are achievable in two dimensions. The classes are: orthorhombic centred
($\gamma = \pi/4$), hexagonal, and tetragonal. A primitive cell and unit
cell are shown in black (note the primitive cell and the unit cell are
identical for the tetragonal case). The coloured regions represent areas
where the acoustic radiation potential is less than $\lambda_{\min} +
0.1(\lambda_{\max}-\lambda_{\min})$ and $\lambda_{\min}$
(resp. $\lambda_{\max}$)
are the minimum (resp. maximum) eigenvalues of $\mQ(\vzero)$. The
blue regions are the expected locations of minima due to the prescribed
minimum location (the origin). The red regions are other minima that
appear in the process. In all figures, the acoustic radiation potential
parameters are $\mathfrak{a}=\mathfrak{b}=1$. The black arrows indicate
the directions normal to the transducer surfaces.}
\label{fig:bravais2d}
\end{figure}

%%%%%%%%%%%%%%%%%%%%%%%%%%%%%%%%%%%%%%%%%%%%%%%%%%%%%%%%%%%%%%%%%%%%%%%%%
%\subsection{Possible crystal-like structures in dimension 3}
%\label{sec:dim3}
%
%Again making the distinction between ARP Bravais lattices and crystal
%Bravais lattices, we examine which of the 14 3D Bravais lattices can be
%obtained for the ARP using the methods described in \ref{sec:arp}.
%Column 1 in Table \ref{tab:bravais3d} lists the 14 3D Bravais lattices.
%We refer to Table 3.3 in \cite{Bradley:2010:MTSS} for the reciprocal vectors of the 3D Bravais lattices, and introduce the constraint $\|\vg_1\| = \|\vg_2\| = \|\vg_3\|$ to obtain the reciprocal vectors in column 2 of \ref{tab:bravais3d}. For some Bravais lattice classes, this constraint implies we cannot generate an ARP Bravais lattice that lies exclusively within the desired class. Just as in the 2D case, any implied symmetries are listed in column 3.
%
\begin{table}
\begin{center}
\begin{tabular}{|c|c|c|}
 \hline
Bravais lattice class & Reciprocal lattice vectors & Implied symmetry \\
 \hline
 Triclinic primitive & $|\vg_1| = |\vg_2| = |\vg_3|$ & \\
 \hline
 \multirow{3}{10em}{} & $\vg_1 = (-\cos \gamma, -\sin \gamma, 0)$  & Cubic primitive (if $\cos \gamma = 0$) \\
 \makecell{Monoclinic primitive} & $\vg_2 = (1,0,0)$ & Tetragonal body-centred \\
 & $\vg_3 = (0,0,1)$ & (if $\cos\gamma \ne 0$) \\
 \hline
 \multirow{3}{4em}{} & $\vg_1 = (-\cot \gamma, -1,0)$ & \\
 \makecell{Monoclinic base-centred} & $\vg_2 = \frac{1}{ac}(c\csc\gamma,0,-a)$ & Tetragonal body-centred \\
 & $\vg_3 = \frac{1}{ac}(c\csc\gamma,0,a)$ & \\
 \hline
 \multirow{3}{4em}{} & $\vg_1 = (0,-1,0)$ & \\
 \makecell{Orthorhombic primitive} & $\vg_2 = (1,0,0)$ & Cubic primitive \\
 & $\vg_3 = (0,0,1)$ &\\
 \hline
 \multirow{3}{4em}{} & $\vg_1 = (b,-a,0)$ & Tetragonal body-centred \\
 \makecell{Orthorhombic base-centred} & $\vg_2 = (b,a,0)$ & (if $a\ne b$) \\
 & $\vg_3 = (0,0,\sqrt{a^2+b^2})$ & Cubic primitive (if $a=b$) \\
 \hline
 \multirow{3}{4em}{} & $\vg_1 = (1,0,1)$ & \\
 \makecell{Orthorhombic body-centred} & $\vg_2 = (0,-1,1)$ & Cubic body-centred \\
 & $\vg_3 = (1,-1,0)$ &\\
 \hline
 \multirow{3}{4em}{} & $\vg_1 = (1/a, 1/b, 1/c)$ & \\
 \makecell{Orthorhombic face-centred} & $\vg_2 = (-1/a, -1/b, 1/c)$ & \\
 & $\vg_3 = (1/a, -1/b, -1/c)$ &\\
 \hline
 \multirow{3}{4em}{} & $\vg_1 = (1,0,0)$ & \\
 \makecell{Tetragonal primitive} & $\vg_2 = (0,1,0)$ & Cubic primitive \\
 & $\vg_3 = (0,0,1)$ &\\
 \hline
 \multirow{3}{4em}{} & $\vg_1 = (0,1,1)$ & \\
 \makecell{Tetragonal body-centred} & $\vg_2 = (1,0,1)$ & Cubic body-centred\\
 & $\vg_3 = (1,1,0)$ &\\
 \hline
 \multirow{3}{4em}{} & $\vg_1 = (0, -2/(3a), 1/(3c))$ & \\
 \makecell{Trigonal primitive} & $\vg_2 = (1/(\sqrt{3}a), 1/(3a), 1/(3c))$ &  \\
 & $\vg_3 = (-1/(\sqrt{3}a), 1/(3a), 1/(3c))$ &\\
 \hline
 \multirow{3}{4em}{} & $\vg_1 = (1/(\sqrt{3}),-1,0)$ & \\
 \makecell{Hexagonal primitive} & $\vg_2 = (2/(\sqrt{3}),0,0)$ & Tegragonal body-centred \\
 & $\vg_3 = (0,0,2/(\sqrt{3}))$ &\\
 \hline
 \multirow{3}{4em}{} & $\vg_1 = (1,0,0)$ & \\
 \makecell{Cubic primitive} & $\vg_2 = (0,1,0)$ & \\
 & $\vg_3 = (0,0,1)$ &\\
 \hline
 \multirow{3}{4em}{} & $\vg_1 = (-1,1,1)$ & \\
 \makecell{Cubic face-centred} & $\vg_2 = (1,-1,1)$ & \\
 & $\vg_3 = (1,1,-1)$ &\\
 \hline
 \multirow{3}{4em}{} & $\vg_1 = (0,1,1)$ & \\
 \makecell{Cubic body-centred} & $\vg_2 = (1,0,1)$ & \\
 & $\vg_3 = (1,1,0)$ &\\
 \hline
 \end{tabular}
 \end{center}
 \caption{Three dimensional Bravais lattice classes that are achievable
 using standing acoustic waves. The parameters $a,b$, and $c$ are the
 lengths of the sides of the conventional unit cell for the Bravais
 lattice when they can be adjusted. The reciprocal lattice vectors we
 give satisfy $|\vg_1| = |\vg_2| = |\vg_3|$ and need to be rescaled so
 that they have length $k$ in order for them to be the wavevectors
 needed to realize a particular Bravais lattice. The angle $\gamma$ is the
 angle between two of the sides of the conventional unit cell.}
 \label{tab:bravais3d}
\end{table}

\hspace{-3cm}

\begin{figure}
\begin{center}
\begin{tabular}{ccc}
\includegraphics[scale=0.14]{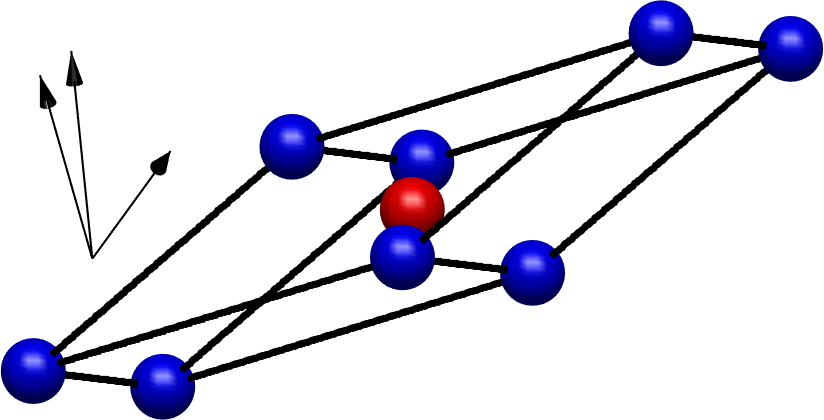} &
\includegraphics[scale=0.15]{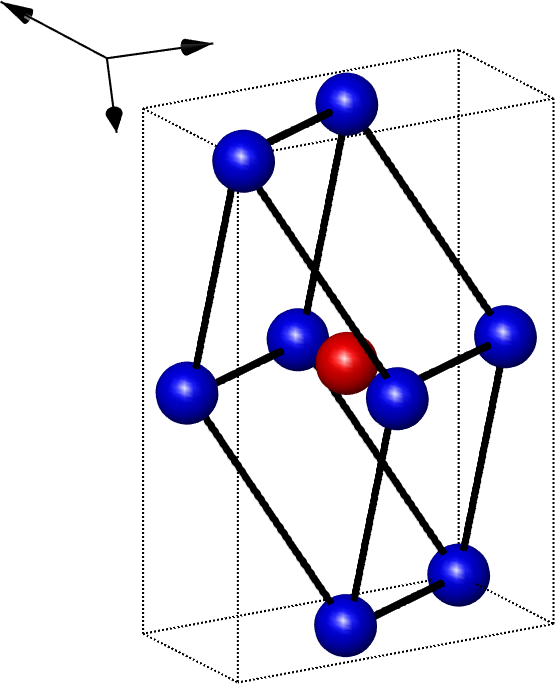} &
\includegraphics[scale=0.1]{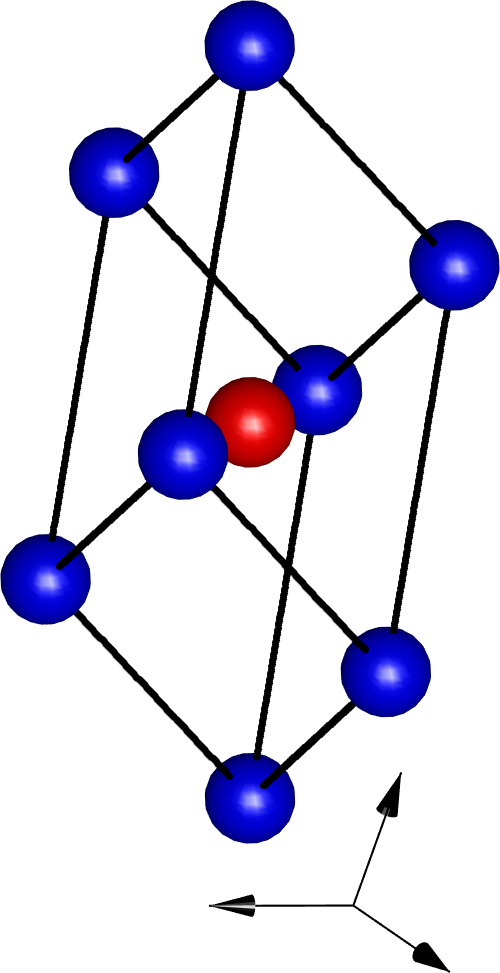} \\
Triclinic primitive & Orthorhombic face-centred & Trigonal primitive \\

\includegraphics[scale=0.1]{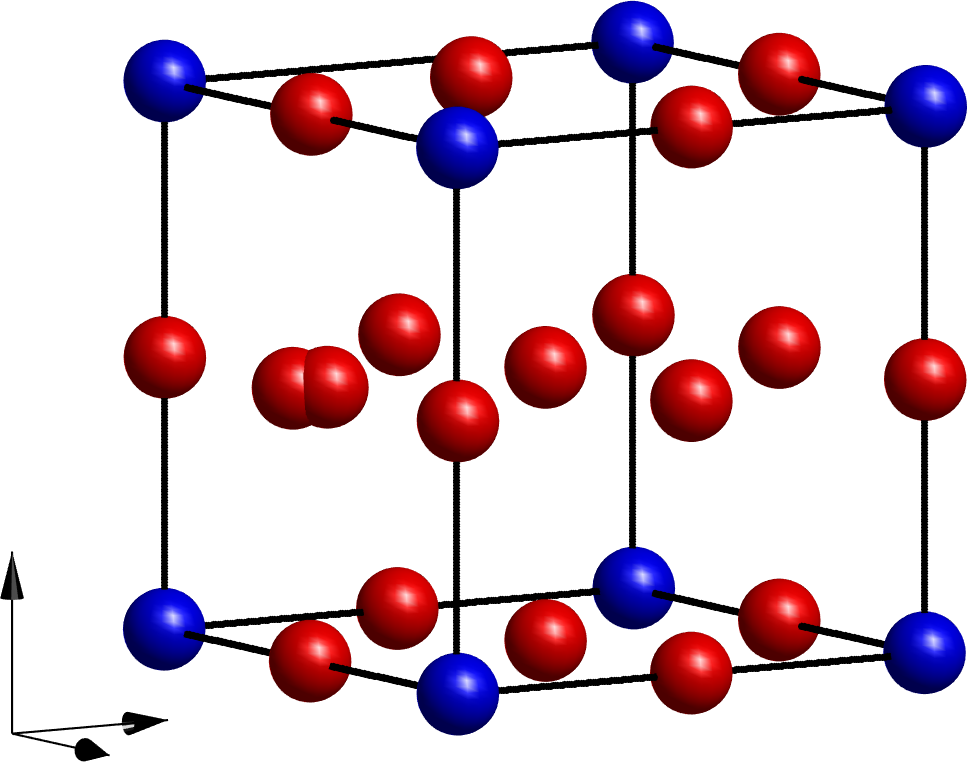} &
\includegraphics[scale=0.1]{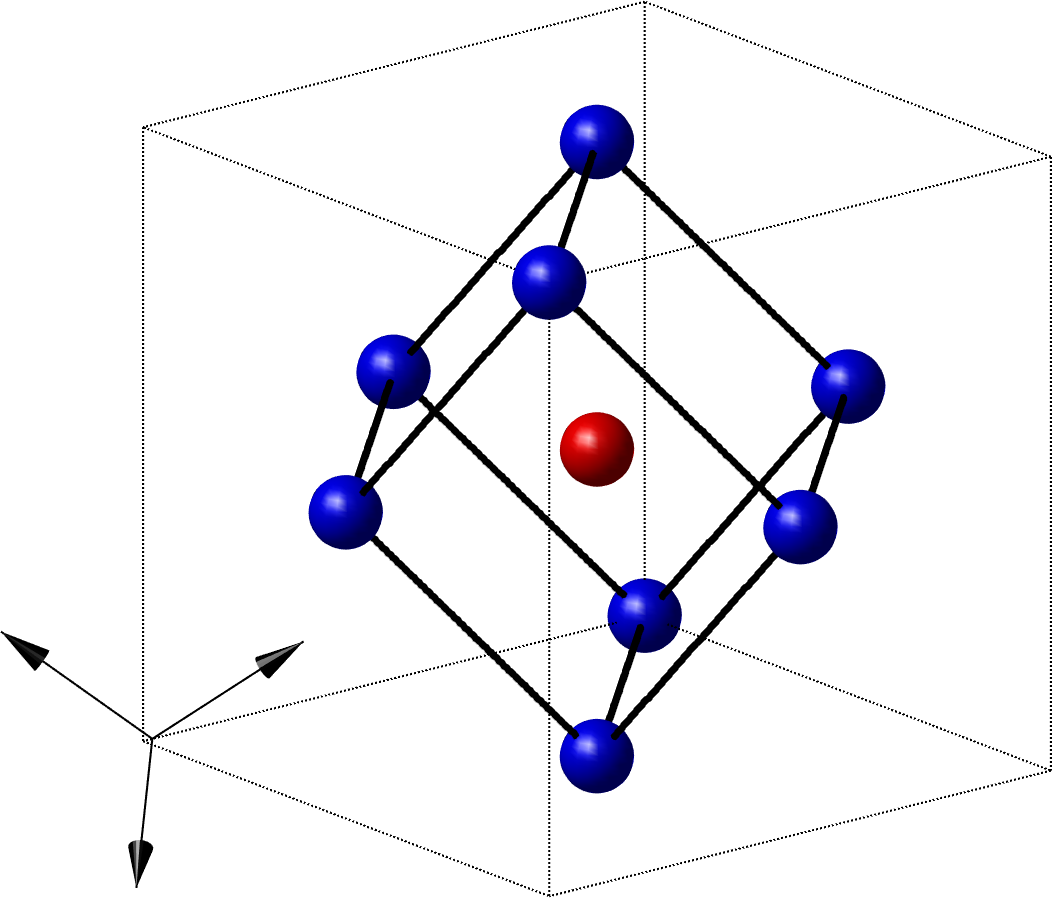} &
\includegraphics[scale=0.12]{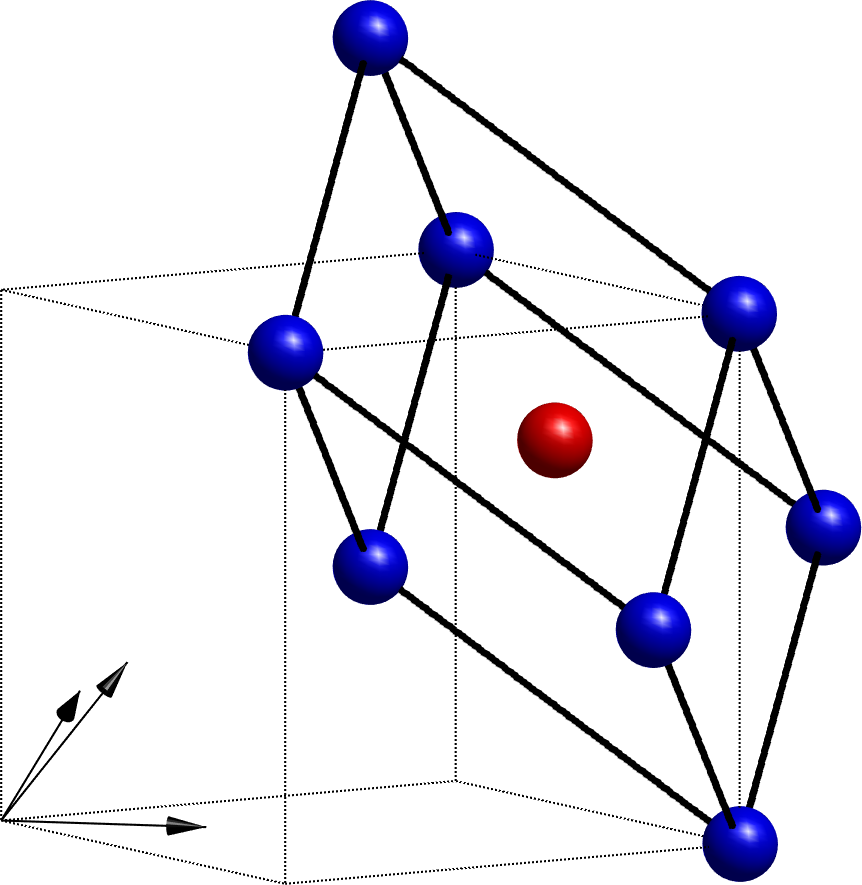} \\
Cubic primitive & Cubic face-centred & Cubic body-centred
\end{tabular}
\end{center}
\caption{Representatives of the six achievable 3D Bravais lattice
classes. The classes are: triclinic primitive ($\vg_1 = (1,2,7), \vg_2 =
(8,3,5), \vg_3 = (1,3,5)$), orthorhombic face-centred ($a=1, b=2, c=3$),
trigonal primitive ($a=1, c=2$), cubic
primitive, cubic face-centred, and cubic body-centred. The primitive
cell is shown in black. The blue regions are the expected locations of
minima due to the prescribed minimum location (the origin). The red
regions are other minima that appear in the process. In all figures, the
acoustic radiation potential parameters are $\mathfrak{a}=\mathfrak{b}=1$.
The black arrows indicate the directions normal to the transducer
surfaces.}
\label{fig:bravais3d}
\end{figure}

%%%%%%%%%%%%%%%%%%%%%%%%%%%%%%%%%%%%%%%%%%%%%%%%%%%%%%%%%%%%%%%%%%%%%%%%
\section{Summary and perspectives}
\label{sec:summary}
We have shown that the behaviour of a periodic acoustic radiation
potential in two or three dimensions can be characterized by a $4\times
4$ or $6 \times 6$ real symmetric matrix, whose eigendecomposition can be
found explicitly. Our main result is to use symmetries of the corresponding
eigenspaces to predict whether the global minima of the acoustic radiation
potential are limited to points, lines or planes. It is still an open
question whether the global minima for periodic acoustic radiation
potentials can be more general curves or surfaces. We also identify
classes of Bravais lattices that can be achieved using a periodic
acoustic radiation potential. We are currently working on extending this
work to quasi-periodic arrangements of particles, which correspond to the
case where we use a number of transducer directions that is larger than
the dimension $d$. Notice that the acoustic radiation potential
derivation assumes the particles are spherical. Thus it is also interesting
to see how the particle shape influences the possible periodic
arrangements of particles. We also plan to characterize the scattering
off gratings for millimetre waves that can be fabricated using
ultrasound directed self-assembly.

%%%%%%%%%%%%%%%%%%%%%%%%%%%%%%%%%%%%%%%%%%%%%%%%%%%%%%%%%%%%%%%%%%%%%%%%

\bibliographystyle{abbrv}
\bibliography{herglotz_bib}

\end{document}